\documentclass[12pt,reqno]{amsart} \usepackage[margin=1.9cm]{geometry}\openup .6em

\newcommand{\todo}[1] {} 

\usepackage{amsmath,amssymb,amsfonts,amstext,verbatim,amsthm,mathrsfs}
\usepackage[mathcal]{eucal}
\usepackage[all]{xy}
\usepackage{xcolor}

\theoremstyle{plain}
\newtheorem{thm}{Theorem}[section]

\newtheorem{lemma}[thm]{Lemma}

\newtheorem{defn}[thm]{Definition}
\newtheorem{definition}[thm]{Definition}
\theoremstyle{remark}
\newtheorem{remark}[thm]{Remark}

\makeatletter
\def\makeCal#1{\expandafter\newcommand\csname c#1\endcsname{\mathcal{#1}}}
\def\makeBB#1{\expandafter\newcommand\csname b#1\endcsname{\mathbb{#1}}}
\def\makeFrak#1{\expandafter\newcommand\csname f#1\endcsname{\mathfrak{#1}}}
\count@=0 \loop \advance\count@ 1 \edef\y{\@Alph\count@}
\expandafter\makeCal\y \expandafter\makeBB\y \expandafter\makeFrak\y
\ifnum\count@<26 \repeat

\newcommand{\ignore}[1]{}

\renewcommand{\comment}[1] {}

\newcommand {\id}{\operatorname{id}}
\newcommand{\im}{\operatorname{im}}

\newcommand {\Hom}{\operatorname{Hom}}
\newcommand{\End}{\operatorname{End}}
\newcommand {\Aut}{\operatorname{Aut}}

\newcommand{\half}{\tfrac{1}{2}}
\newcommand{\tensor}{\otimes}
\renewcommand{\O}{\mathscr{O}}
\newcommand{\isom}{\cong}

\newcommand{\mat}[4]{\begin{pmatrix}#1&#2\\#3&#4\end{pmatrix}}

\newcommand{\hk}{hyperk{\"a}hler }

\newcommand{\lra}{\longrightarrow}
\newcommand{\lRa}[1]{\stackrel{#1}{\longrightarrow}}
\newcommand{\Lie}{\cL}

\newcommand{\PGL}{\operatorname{PGL}}
\newcommand{\MCG}{\operatorname{MCG}}
\newcommand{\hash}{\#}
\renewcommand{\leq}{\leqslant}

\newcommand{\PT}{\operatorname{PT}}
\newcommand{\GL}{\operatorname{GL}}

\begin{document}

\title{Joyce structures and their twistor spaces}

\author{Tom Bridgeland}

\begin{abstract}
Joyce structures are a class of geometric structures which first arose in relation to holomorphic generating functions for Donaldson-Thomas invariants. They can be thought of as non-linear analogues of Frobenius structures, or as special classes of complex  \hk manifolds. We give a detailed introduction to Joyce structures, with particular focus on the geometry of the associated twistor space. We also prove several new results. 
 \end{abstract}


 \maketitle

\section{Introduction}

The aim of this paper is to give an  introduction to a class of geometric structures known as Joyce structures.  These structures have appeared in several contexts recently, including integrable systems  \cite{A2,D,DM2} and  topological string theory \cite{AP,AP2}. A Joyce structure on a complex manifold $M$ involves a one-parameter family of flat and symplectic  non-linear connections on the tangent bundle $T_M$, and gives rise to a complex \hk structure on the total space $X=T_M$. The precise definition\todo{Mention this is a strong Joyce sstructure in the language of \cite{Strachan}}  first  appeared in \cite{RHDT2}, but the essential features are standard in  twistor theory (see e.g. \cite{DM}), and go back to work of  Pleba{\'n}ski \cite{P}.

Joyce structures take their name from a line of research initiated in  \cite{holgen} which aims to encode the Donaldson-Thomas (DT)  invariants \cite{JS,KS1} of a three-dimensional Calabi-Yau (CY$_3$) category in a geometric structure on its space of stability conditions \cite{Stab1,Stab2}.   From this point of view a Joyce structure should be thought of as a non-linear analogue of a Frobenius structure \cite{Dub1,Dub2}, in which the linear structure group $\GL_n(\bC)$ has been replaced by the group of Poisson automorphisms of a complex torus $(\bC^*)^n$. The wall-crossing formula shows that the DT invariants can be viewed as the Stokes data for an isomonodromic family of irregular connections on $\bP^1$ taking values in this group. This perspective is the subject of \cite{RHDT2}  and is summarised in \cite[Appendix A]{Strachan}.

Relations between the wall-crossing formula in DT theory and real \hk manifolds  were first discovered in the celebrated work of Gaiotto, Moore and Neitzke \cite{GMN1,GMN2}. The connection with complex  \hk manifolds is somewhat different and was made  in \cite{Strachan}. In physical terms the two stories are related by the conformal limit \cite{AP,AP2,G}.  

The geometry of a Joyce structure  is often clearer when viewed through the lens of the associated twistor space $p\colon Z\to \bP^1$. Each fibre of the map $p$ is the leaf space of a half-dimensional foliation on $X=T_M$. There are essentially three distinct fibres, $Z_0$, $Z_1$ and $Z_\infty$, of which $Z_0$ is naturally identified with  $M$. It is intriguing to note that in simple examples associated to the DT theory of a quiver, the fibre $Z_0=M$ is a quotient of the space of stability conditions, whereas $Z_1$ is closely related to the  cluster Poisson variety. The fibre $Z_\infty$ remains rather mysterious.

An interesting  class of examples of Joyce structures was constructed in \cite{CH}.  The base $M$ parameterises pairs consisting of a Riemann surface of some fixed genus $g$ equipped with a quadratic differential  with simple zeroes.  The extension to spaces of quadratic differentials with poles of fixed orders will appear in \cite{Z}. These Joyce structures are expected to arise from the DT theory of the  CY$_3$ categories considered in \cite{BS, H}, and relate in physics \cite{BLMST} to supersymmetric gauge theories of class $S[A_1]$. We give explicit descriptions of the two simplest examples of this type, which are naturally associated to the DT theory of the A$_1$ and A$_2$ quivers respectively. The A$_2$  example is particularly interesting, and is closely related to the  Painlev{\'e} I equation.

\subsection*{Plan of the paper}  We begin in Section \ref{general} with the notion of a pre-Joyce structure on a  complex manifold $M$. It gives rise to a complex \hk structure on the total space $X=T_M$. A Joyce structure is a pre-Joyce structure with certain additional symmetries. These symmetries are discussed in Section \ref{further} and are controlled by a special type of integral affine structure on $M$ which we  call  a period structure.\footnote{Some of this material appears also in \cite{Strachan} but we have chosen to tell the story again from the beginning because a more detailed treatment of several points seems worthwhile, and experience with examples has suggested a few small changes in the definitions.}

The twistor space $p\colon Z\to \bP^1$ associated to a Joyce structure is introduced in Section \ref{twist}. In Section \ref{letstwistagain} we establish several new results on its structure.
In Section \ref{ham} we show how to associate a strongly-integrable Hamiltonian system to a Joyce structure equipped with certain extra data, namely a Lagrangian submanifold $R\subset Z_\infty$, and an identification of $M$ with an open subset of the cotangent bundle of a complex manifold $B$.

In the rest of the paper we discuss examples. In Section \ref{geometric} we first recall from \cite{CH} the construction of Joyce structures of class $S[A_1]$ on spaces of holomorphic quadratic differentials. We then explain how some of the constructions from previous sections play out in this setting.
In Section \ref{a1} and \ref{a2} we give explicit formulae in the two simplest cases, which are  associated to the DT theory of the A$_1$ and A$_2$ quivers respectively. 
 
\subsection*{Conventions } We work throughout in the category of complex manifolds and holomorphic maps. All symplectic forms, metrics, bundles, connections, sections etc. are  holomorphic. The holomorphic tangent bundle of a complex manifold $M$ is denoted $T_M$, and the derivative of a map of complex manifolds $f\colon M\to N$ is denoted $f_*\colon T_M\to f^*(T_N)$. The map $f$ is called {\'e}tale if $f_*$ is an isomorphism. We use the symbol $\cL$ to denote the Lie derivative.

\subsection*{Acknowledgements} The author is very grateful for conversations and correspondence with Sergey Alexandrov, Anna Barbieri, Maciej Dunajski, Dominic Joyce, Dimitry Korotkin, Davide Masoero, Boris Pioline, Ian Strachan, J{\"o}rg Teschner and Iv{\'a}n Tulli.

\section{Pre-Joyce structures}
\label{general}

Recall \cite{Dub1,Dub2} that a Frobenius structure on a complex manifold $M$ consists of a pencil of flat, torsion-free connections on the tangent bundle of $M$  with certain additional properties (the existence of the identity and Euler vector fields, potentiality of the multiplication, etc). In a similar way,  a Joyce structure on $M$ consists of a pencil of flat, symplectic \emph{non-linear} connections on the tangent bundle of $M$, again admitting certain additional symmetries. In this section we focus on the pencil of  connections, which we refer to as a pre-Joyce structure, leaving discussion of the required  symmetries to the next section. 

\subsection{Non-linear connections}
\label{connect}

We begin by briefly summarising some basic facts about non-linear  connections in the sense of Ehresmann (see e.g. \cite[Chapter 3]{new}). Recall that we are always working in the category of complex manifolds, so all connections will be holomorphic. 
 
Let $\pi\colon X\to M$ be a holomorphic submersion of complex manifolds. We denote the fibres by $X_m=\pi^{-1}(m)$. The derivative of $\pi$ gives rise to a  short exact sequence of bundles
\begin{equation}\label{bass}0\lra T_{X/M}\lRa{i} T_X\lRa{\pi_*} \pi^*(T_M)\lra 0.\end{equation}
\begin{defn}
A non-linear
connection on  the map $\pi$ is a   map of bundles $h\colon \pi^*(T_M)\to T_X$ satisfying $\pi_*\circ h=1$.\end{defn}
Writing $H=\im(h)$ and $V=T_{X/M}$, the tangent bundle of $X$ decomposes as a direct sum $T_X=H\oplus V$. We call tangent vectors and vector fields horizontal or vertical if they lie in $H$ or $V$ respectively. 
Note that a vector field $u\in H^0(M,T_M)$ can be lifted to a horizontal vector field $h(u)\in H^0(X,T_X)$ by composing  the pullback $\pi^*(u)\in H^0(X,\pi^*(T_M))$ with the map $h$.

Consider a smooth path $\gamma\colon [0,1]\to M$. Given a point  $x\in X_{\gamma(0)}$ we can look for a lifted path  $\alpha\colon [0,\delta]\to X$ satisfying $\alpha_*(\frac{d}{dt})=h(\gamma_*(\frac{d}{dt}))$ and $\alpha(0)=x$. Since we have not assumed that $\pi$ is proper, such a  lift will exist only for small enough $\delta>0$. For $t\in [0,\delta]$ we call $\alpha(t)\in X_{\gamma(t)}$ the time $t$ parallel transport of the point $x$ along the path $\gamma$. Given a point $x_0\in X_{\gamma(0)}$ we can find a $\delta>0$ and open subsets  $U_t\subset X_{\gamma(t)}$ with $x_0\in U_0$, such that time $t$ parallel transport along $\gamma$ defines an  isomorphism $\PT_{\gamma}(t)\colon U_0\to U_t$ for each $t\in [0,\delta]$.

Given  complex manifolds $M,N$ there is a connection on the projection map $\pi_M\colon M\times N\to M$ induced by the canonical splitting $T_{M\times N}=\pi_M^*(T_M)\oplus \pi_N^*(T_N)$.  The connection $h$ is called flat if it is locally isomorphic to a connection of this form. More precisely:

\begin{definition}
The connection $h$ is flat if the following equivalent conditions hold:
\begin{itemize}
\item[(i)] for every $x\in X$ there are local co-ordinates  $(x_1,\cdots, x_n)$ on $X$ at $x$, and $(y_1,\cdots, y_d)$ on $M$ at  $\pi(x)$, such that $x_i=\pi^*(y_i)$ and  $h(\frac{\partial}{\partial y_i})=\frac{\partial}{\partial x_i}$ for $1\leq i\leq d$,
\item[(ii)]  
the sub-bundle $H=\im(h)\subset T_X$ is closed under Lie bracket: $[H,H]\subset H$.
\end{itemize}
\end{definition}

Suppose given a relative symplectic form $\Omega_{\pi}\in H^0(X,\wedge^2\,\, T^*_{X/M})$ on the map $\pi$. It restricts to a symplectic form $\Omega_m\in H^0(X_m,\wedge^2\,\, T^*_{X_m})$ on each fibre $X_m$. 
 We say that the connection $h$ preserves $\Omega_{\pi}$ if for any path $\gamma\colon [0,1]\to M$  the partially-defined parallel transport maps $\PT_{\gamma}(t)\colon X_{\gamma(0)}\to X_{\gamma(t)}$ take $\Omega_{\gamma(0)}$ to $\Omega_{\gamma(t)}$. This is equivalent to the statement that for any horizontal vector field $u$ on $X$ the Lie derivative $\cL_{u}(\Omega_\pi)=0$.

Using the decomposition $T_X=H\oplus V$ the relative form $\Omega_{\pi}$ can be lifted uniquely to a form $\Omega\in H^0(X,\wedge^2\,  T^*_{X})$ satisfying $\ker(\Omega)=H$.
\begin{lemma}
\label{symp}
\begin{itemize}
\item[(i)] The connection $h$ preserves $\Omega_\pi$ precisely if $i_{v_1} i_{v_2} (d\Omega)=0$ for any two vertical vector fields $v_1,v_2\in H^0(X,T_{X/M})$.
\item[(ii)] If the connection $h$ is flat then it preserves $\Omega_\pi$ precisely if $d\Omega=0$.
\end{itemize}
 \end{lemma}
 
 \begin{proof}
 Part (i) is  \cite[Theorem 4]{Gotay}.\todo{But note it doesn't seem to be actually proved there, and anyway that's for real symplectic forms.} For part (ii), note that one implication follows from (i), so let us assume that $h$ is flat and preserves $\Omega_\pi$ and prove that $d\Omega=0$.
 
 Take  three vector fields $u_1,u_2,u_3$ on $X$ and consider the expression defining $d\Omega(u_1,u_2,u_3)$. We can assume that each $u_i$ is either horizontal or vertical. Note that  both horizontal and vertical vector fields are closed under  Lie bracket, and $i_h(\Omega)=0$ for any horizontal vector field $h$. Thus   $d\Omega(u_1,u_2,u_3)=0$ as soon as two of the $u_i$ are horizontal. In the remaining cases two of the $u_i$ are vertical, and the claim follows from part (i).\end{proof}

Suppose that a discrete group $G$ acts freely and properly on  $X$ preserving the map $\pi$. Then $Y=X/G$ is a complex manifold and the quotient map $q\colon X\to Y$ is {\'e}tale. There is an induced submersion $\eta\colon Y\to M$ and a factorisation $\pi=\eta\circ q$. A connection $h\colon \pi^*(T_M)\to T_X$  will be called $G$-invariant  if $g_*\circ h=h$ for all $g\in G$. There is then an induced connection $j\colon \eta^*(T_M)\to T_Y$ on $\eta$  defined uniquely by the condition that $q_*\circ h =q^*(j)$. We say that the connection $h$ descends along the quotient map $q$.

\subsection{Pre-Joyce structures}
\label{prejoyce}

Let $M$ be a complex manifold and let $\pi\colon X=T_M\to M$ be the total space of the tangent bundle of $M$. There is  a canonical isomorphism $\nu\colon \pi^*(T_M)\to T_{X/M}$  obtained by composing the chain of identifications
\begin{equation}\pi^*(T_M)_x=T_{M,\pi(x)}=  T_{T_{M,\pi(x)},x}= T_{X_{\pi(x)},x}=T_{X/M,x}.\end{equation}
 We set $v=i\circ\nu$.  A connection $h\colon \pi^*(T_M)\to T_X$ on $\pi$ then defines a family of  connections $h_\epsilon=h+\epsilon^{-1} v$ parameterised by $\epsilon\in \bC^*$.\footnote{At this stage it might seem more sensible to parameterise the pencil by $t=\epsilon^{-1}$. In later applications however it is the parameter $\epsilon$ which appears most naturally.}  

\begin{equation}
\xymatrix@C=1em{  
 0\ar[rr] && T_{X/M}  \ar[rr]^{i} &&T_X  \ar[rr]^{\pi_*} &&\pi^*(T_M) \ar@/_1.8pc/[ll]_{h_\epsilon} \ar@/^1.8pc/[llll]^\nu \ar[rr] && 0 } \end{equation}

Suppose that $M$ is equipped with a holomorphic symplectic form $\omega\in H^0(M,\wedge^2 \, T^*_M)$. Via the isomorphism $\nu$ we obtain a  relative symplectic form $\Omega_\pi\in H^0(X,\wedge^2 T_{X/M}^*)$ which restricts to  a linear symplectic form $\omega_m$ on each fibre $X_m=T_{M,m}$.
We say that a connection on $\pi$ is symplectic if it preserves $\Omega_\pi$ in the sense defined above.  

\begin{defn}
A pre-Joyce structure $(\omega,h)$ on a complex manifold $M$ consists of
\begin{itemize}
\item[(i)] a {holomorphic symplectic form} $\omega$ on $M$,
\item[(ii)] a {non-linear connection} $h$ on the tangent bundle $\pi\colon X=T_M\to M$,
\end{itemize}
such that for each $\epsilon\in \bC^*$ the connection $h_\epsilon=h+\epsilon^{-1}v$ is flat and symplectic.\end{defn}

Take a local co-ordinate system  $(z_1,\cdots,z_n)$ on $M$ which is Darboux, in the sense that
\begin{equation}
\label{wpq}\omega = \frac{1}{2}\cdot \sum_{p,q} \omega_{pq} \cdot dz_p\wedge dz_q,\end{equation}
with $\omega_{pq}$ a constant skew-symmetric matrix.  We denote by $\eta_{pq}$ the inverse matrix.

There are  associated linear co-ordinates $(\theta_1,\cdots,\theta_n)$ on the tangent spaces $T_{M,m}$  obtained by writing a tangent vector in the form $\sum_i \theta_i\cdot  {\partial}/{\partial z_i}$. We thus get  induced local  co-ordinates $(z_i,\theta_j)$ on  the total space $X=T_M$. In these co-ordinates 
\begin{equation}\label{v}v_i:=v\Big(\frac{\partial}{\partial z_i}\Big)=\frac{\partial}{\partial \theta_i}.\end{equation}


The fact that the connection $h$ is flat and symplectic ensures that we can write\begin{equation}
\label{above}h_i:=h\Big(\frac{\partial}{\partial z_i}\Big)= \frac{\partial}{\partial z_i} + \sum_{p,q} \eta_{pq} \cdot \frac{\partial W_i}{\partial \theta_p} \cdot \frac{\partial}{\partial \theta_q},\end{equation}
for locally-defined functions $W_i=W_i(z,\theta)$. Note that $W_i$ is only well-defined up to the addition of functions $a_i(z)$ independent of the $\theta$ co-ordinates. We can fix these integration constants by insisting that $W_i$ vanishes along the zero section $M\subset X=T_M$, i.e. that
\begin{equation}\label{fix1} W_i(z_1,\cdots,z_n,0,\cdots,0)=0.
\end{equation}

The connection $h_\epsilon$ is  flat precisely if $\big[h_i+\epsilon^{-1} v_i,h_j+\epsilon^{-1} v_j\big]=0$ for all $1\leq i,j\leq n$. A short calculation shows that this holds for all $\epsilon\in \bC^*$ precisely if

 \begin{gather}\label{southgate}
     \frac{\partial}{\partial \theta_k}\Big(\frac{\partial W_i}{\partial \theta_j}-\frac{\partial W_j}{\partial \theta_i}\Big)=0, \\
\label{jude}
    \frac{\partial}{\partial \theta_k}\Big(\frac{\partial  W_i}{\partial z_j}-\frac{\partial W_j}{\partial z_i }-\sum_{p,q} \eta_{pq} \cdot \frac{\partial W_i}{\partial \theta_p} \cdot \frac{\partial W_j}{\partial \theta_q}\Big)=0.
 \end{gather}
for all $1\leq i,j,k\leq n$.

\begin{remark}So as to be able to include certain interesting examples it is sometimes useful to weaken the axioms of a (pre-) Joyce structure to allow the connection  $h\colon \pi^*(T_M)\to T_X$  to have poles.  In precise terms this means that   $h$ should be defined by a bundle map $h\colon \pi^*(T_M)\to T_X(D)$ satisfying \begin{equation}(\pi_*\tensor \O_X(D))\circ h=1_{\pi^*(T_M)}\tensor s_D,\end{equation}
 where  $D\subset X$ is an effective divisor, and $s_D\colon \O_X\to \O_X(D)$ is the canonical inclusion. When expressed in  terms of local co-ordinates as above, this just means that the   functions $W_i(z,\theta)$ are meromorphic. We will refer to the resulting structures as meromorphic (pre-) Joyce structures.
 \end{remark} 
 
\subsection{Complex \hk structures}
\label{comphk}

The following  structures have appeared before in the literature under various different names (see e.g. \cite{BP,D, JV1,JV2}). We emphasise that,  as throughout the paper, all quantities appearing are holomorphic.  
\begin{definition} 
\label{complexhk}
A complex \hk structure $(g,I,J,K)$ on  a complex manifold $X$ consists of
\begin{itemize}
\item[(i)] a metric $g$, i.e. a symmetric, non-degenerate bilinear form $g\colon T_X \tensor T_X \to \O_X$,
\item[(ii)] endomorphisms $I,J,K\in \End_X(T_X)$,
\end{itemize}
such that:
\begin{itemize}
\item[(HK1)]  the quaternion relations hold:
$I^2=J^2=K^2=IJK=-1$,
\item[(HK2)]$I,J,K$ are compatible with  $g$ and  are parallel for the associated  Levi-Civita connection $\nabla$:
\begin{equation}g(R (u_1), R(u_2))= g(u_1,u_2) , \qquad \nabla(R)=0, \qquad R\in \{I,J,K\}.\end{equation}
\end{itemize}
\end{definition}

Let $M$ be a complex manifold with a holomorphic symplectic form $\omega$. A non-linear connection $h$ on the tangent bundle $\pi\colon X=T_M\to M$ gives a decomposition\begin{equation}T_X=\im(h)\oplus \im(v) \isom \pi^*(T_M)\tensor_{\bC}\bC^2.\end{equation}
We can then define an action of the quaternions on $T_X$ by choosing an identification of the complexification of the quaternions $\bH\tensor_{\bR}\bC$  with the algebra of $2\times 2$ matrices $\End_{\bC}(\bC^2)$. We can also define a metric $g$  by taking the  tensor product of  $\pi^*(\omega)$ with a linear symplectic form on $\bC^2$.

With appropriate conventions this leads to the formulae\footnote{Compared to \cite{Strachan} we have changed the signs of $I$ and $K$, and divided the metric by 2.}

\begin{equation}
\label{ijk}
\begin{aligned}
I\circ h&=i\cdot h,\qquad\quad & J\circ h&=-v,  &\qquad\quad  K\circ h&=i\cdot  v, \\
I\circ v&=-i\cdot v,   &  J\circ v&=h,   & K\circ v&=i \cdot h, \end{aligned}\end{equation}
which should be interpreted as equalities of maps $\pi^*(T_M)\to T_X$, and
\begin{equation}\label{g}g(h(u_1),v(u_2))=\half \omega(u_1,u_2), \qquad g(h(u_1),h(u_2))=0= g(v(u_1),v(u_2)).
\end{equation}
It is easily checked that $g$ is preserved by the endomorphisms $I,J,K$.

The following result implies in particular that a pre-Joyce structure on a complex manifold $M$  induces a complex \hk structure on the total space $X=T_M$.

\begin{thm} \label{jon}The   operators $I,J,K$ are parallel for the Levi-Civita connection $\nabla$ associated to $g$ precisely if the connection $h_\epsilon=h+\epsilon^{-1}v $ is flat and symplectic for all $\epsilon\in \bC^*$.\end{thm}

\begin{proof}

We begin with a general remark. Let $g\colon T_X\times T_X\to \O_X$ be a metric on a complex manifold $X$ with associated Levi-Civita connection $\nabla$. Let $R\in \End_X(T_X)$ be an endomorphism which is compatible with $g$ and satisfies $R^2=-1$. We can then define a 2-form $\Omega$ on $X$ by setting $\Omega_R(u_1,u_2)=g(R(u_1),u_2)$. Let $H\subset T_X$ denote the $+i$ eigenbundle of $R$. Then standard proofs from K{\"a}hler geometry apply unchanged in this holomorphic context  to give implications
\begin{equation}\label{ab}\nabla(R)=0\implies [H,H]\subset H, \qquad \nabla(R)=0 \iff d \Omega_R=0.\end{equation}

Return now to  the setting above. For $\epsilon \in \bC^*$ we  introduce the endomorphism  \begin{equation}J_\epsilon=I-i\epsilon^{-1}(J+iK).\end{equation} A simple calculation using the definitions \eqref{ijk} shows that  $J_\epsilon^2=-1$, and that the $+i$ eigenbundle of $J_\epsilon$ coincides with $H_\epsilon=\im(h_\epsilon)$. 

As in Section \ref{prejoyce}, the symplectic form $\omega$ on $M$ induces a relative symplectic form $\Omega_\pi$ on the projection $\pi\colon X\to M$. Moreover, as explained before Lemma \ref{symp}, there is then a unique 2-form $\Omega_\epsilon$ on $X$ satisfying  the conditions\begin{equation}\ker(\Omega_\epsilon)=H_\epsilon, \qquad \Omega_\epsilon(v(u_1),v(u_2))=\omega(u_1,u_2),\end{equation} where $u_1,u_2$ are arbitrary vector fields on  $M$. Another calculation using  \eqref{ijk} and \eqref{g} shows that this form is given explicitly by the formula
\begin{equation}
\Omega_\epsilon =\epsilon^{-2}\cdot \Omega_++2i\epsilon^{-1}\cdot \Omega_I +\Omega_-, \qquad \Omega_\pm=\Omega_{J\pm iK}.\end{equation}

We can now prove the Theorem. Suppose first that $I,J,K$ are parallel. Then  $J_\epsilon$ is parallel for all $\epsilon\in \bC^*$, and applying \eqref{ab} with $R=J_\epsilon$ we find that $[H_\epsilon,H_\epsilon]\subset H_\epsilon$ and hence that $h_\epsilon$ is flat. Since $\Omega_\epsilon$ is also parallel and hence closed, applying Lemma \ref{symp} shows that $h_\epsilon$ is  symplectic.
Conversely suppose that for all $\epsilon \in \bC^*$ the connection $h_\epsilon$ is flat and symplectic.  Then by ~Lemma \ref{symp} again, $d\Omega_\epsilon=0$  for all $\epsilon \in \bC^*$, and this  easily implies that  $d\Omega_R=0$ for $R\in \{I,J,K\}$. By \eqref{ab} we conclude that $I,J,K$ are parallel. 
\end{proof}

\subsection{Associated 2-forms}

The complex \hk structure $(g,I,J,K)$ gives rise to closed 2-forms  on $X$  \begin{equation}
\label{fart}\Omega_I(w_1,w_2)=g(I(w_1),w_2),\qquad \Omega_{\pm}(w_1,w_2)=g((J\pm iK)(w_1),w_2).\end{equation}
Note that $\Omega_I$ is non-degenerate, but since $(J\pm iK)^2=0$ the forms $\Omega_\pm$ have kernels.

Let us express these forms in terms of a local co-ordinate system  $(z,\theta)$ on $X$  as in Section \ref{prejoyce}. We denote by  $(h^i,v^j)$ the basis of covector fields   dual to the basis of vector fields  $(h_i,v_j)$ defined by \eqref{v} -- \eqref{above}. Thus  $(h^j,v_i)=0=(v^j,h_i)$ and $(h^j,h_i)=\delta_{ij}=(v^j,v_i)$. Explicitly we have  \begin{equation}
\label{above2}h^j= dz_j, \qquad v^j= d\theta_j + \sum_{r,s} \eta_{jr} \cdot \frac{\partial W_s}{\partial \theta_r} \cdot dz_s.\end{equation}
The definition \eqref{ijk} of the operators $I,J,K$ immediately give

\begin{equation}\label{here}(J+iK)\circ h =-2v,\qquad  (J+iK)\circ v=0=(J-iK)\circ h, \qquad (J-iK)\circ v=2h,\end{equation}

and so  we have

\begin{gather}\Omega_+=\frac{1}{2}\cdot \sum_{p,q} \omega_{pq} \cdot h^p \wedge h^q, \qquad \Omega_I= \frac{i}{2}\cdot  \sum_{p,q} \omega_{pq} \cdot v^p\wedge h^q, \\\Omega_-= \frac{1}{2} \cdot \sum _{p,q} \omega_{pq} \cdot v^p\wedge v^q.  \end{gather}
Using the formulae \eqref{above2} these expressions can be rewritten as

\begin{gather}
\label{simple}\Omega_+=\frac{1}{2}\cdot  \sum_{p,q} \omega_{pq} \cdot dz_p \wedge dz_q, \\
\label{south}2i \Omega_I=\frac{1}{2}\cdot \sum_{p,q} \bigg(\frac{\partial W_p}{\partial \theta_q}-\frac{\partial W_q}{\partial \theta_p} \bigg)\cdot dz_p\wedge dz_q-\sum_{p,q} \omega_{pq} \cdot d \theta_p  \wedge d z_q, \\
\label{w}\Omega_-=\frac{1}{2}\cdot \sum_{p,q}\omega_{pq} \cdot d\theta_p\wedge d\theta_q+\sum_{p,q} \frac{\partial W_q}{\partial \theta_p } \cdot d \theta_p \wedge d z_q -\frac{1}{2} \cdot \sum_{p,q,r,s} \eta_{rs}\cdot \frac{\partial W_p}{\partial \theta_r} \frac{\partial W_q}{\partial \theta_s} \cdot dz_p\wedge dz_q.\end{gather}
Note that by \eqref{southgate} the first term in \eqref{south} is independent of the co-ordinates $\theta_k$ and hence descends to $M$. In the case of a Joyce structure the identity \eqref{nick} shows that this term vanishes.


\section{Joyce structures}
\label{further}
  
A Joyce structure on a complex manifold $M$ is a pre-Joyce structure with certain additional symmetries.  These symmetries are  controlled by a special kind of integral affine structure on $M$ which we  call  a period structure. After introducing the necessary definitions we derive some consequences of the extra symmetries, both for the associated complex \hk structure  on $X=T_M$, and for the local generating functions $W_i$. 

\subsection{Period structures}

Let $\cH$ be a  holomorphic vector bundle on a complex manifold $M$. By a {lattice} in $\cH$ we mean a  locally-constant subsheaf of abelian groups $\cH^{\bZ}\subset \cH$ such that the multiplication map $\cH^{\bZ}\tensor_{\bZ}\O_M\to \cH$ is an isomorphism. There is an induced flat (linear) connection  $\nabla$ on $\cH$ whose  flat sections are $\bC$-linear combinations of the sections of  $\cH^{\bZ}$. 

\begin{defn}
\label{pp}
A period structure $(T_M^{\bZ},\nabla,Z)$ on a complex manifold $M$  consists of
\begin{itemize}
\item[(P1)] a lattice $T_M^{\bZ}\subset T_M$ whose associated flat connection is  denoted $\nabla$,
\item[(P2)]  a vector field $Z\in \Gamma(M,T_M)$ satisfying $\nabla(Z)=\id$.
\end{itemize}
\end{defn}

Given a  point $p\in M$, a basis of the free abelian group $T_{M,p}^{\bZ}$ extends uniquely to a basis of $\nabla$-flat sections $\phi_1,\cdots,\phi_n$ of $T_M$ over a contractible open neighbourhood $p\in U\subset M$. Writing the vector field $Z$ in the form $Z=\sum_i z_i\cdot \phi_i$ then defines holomorphic functions $z_i\colon U\to \bC$.  Given a local co-ordinate system $(u_1,\cdots,u_n)$ on $M$,  condition (P2) shows that
\begin{equation}
    \label{late}
\frac{\partial}{\partial u_j}=\nabla_{\frac{\partial}{\partial u_j}}(Z)=\sum_i\frac{\partial z_i}{\partial u_j} \cdot \phi_i,\end{equation}
for all $1\leq j\leq n$, from which it follows that $(z_1,\cdots,z_n)$ is also a local co-ordinate system. Applying \eqref{late} with $u_j=z_j$ then shows that $\phi_i=\frac{\partial}{\partial z_i}$.  This implies in particular that the connection $\nabla$ is torsion-free.

Recall  \cite{KS} that an integral affine structure on a complex manifold $M$ is a lattice $T_M^{\bZ}\subset T_M$ whose associated flat  connection $\nabla$ is torsion-free . A local co-ordinate system $(z_1,\cdots,z_n)$ is then called integral affine if the tangent vectors $\frac{\partial}{\partial z_i}$ lie in the lattice $T_M^{\bZ}$. Such co-ordinate systems are uniquely defined up to affine transformations of the form $z_i\mapsto \sum_j a_{ij}z_j + v_i$ with $(a_{ij})\in \GL_n(\bZ)$ and $(v_i)\in \bC^n$.

Given a period structure on a complex manifold $M$ we obtain an integral affine structure  by forgetting the vector field $Z$. A system of integral affine co-ordinates $(z_1,\cdots,z_n)$ will be called integral linear if $Z=\sum_i z_i \cdot \frac{\partial}{\partial z_i}$. Such co-ordinate systems are uniquely defined up to linear transformations of the form $z_i\mapsto \sum_j a_{ij}z_j$ with $(a_{ij})\in \GL_n(\bZ)$. Thus a period structure can be thought of as an integral \emph{linear} structure.


\begin{definition}
A period structure will be called homogeneous if the vector field $Z$ generates an action of the multiplicative group $\bC^*$ on $M$.  \end{definition}


We can use the connection $\nabla$ on $T_M$ to lift the vector field $Z$ on $M$ to a horizontal vector field $E$ on $X=T_M$. If we take a system of integral linear co-ordinates $(z_1,\cdots,z_n)$ on $M$, and associated co-ordinates $(z_i,\theta_j)$ on  $X=T_M$ as  in Section \ref{prejoyce}, then
\begin{equation}
    \label{ze}
    Z=\sum_i z_i\cdot \frac{\partial}{\partial z_i}, \qquad E=\sum_i z_i\cdot \frac{\partial}{\partial z_i}.
\end{equation}In what follows we   refer to  both $Z$ and $E$ as Euler vector fields. This will hopefully  not cause confusion since they live on different spaces. 

\begin{lemma}
\label{jen}
Let $(T_M^{\bZ},\nabla,Z)$ be a homogeneous period structure on a complex manifold $M$. Then  the lifted Euler vector field  $E$ generates an action of $\bC^*$ on $X=T_M$.   This action is obtained by combining the
derivative of the action map $a\colon \bC^*\times M\to M$, with a rescaling of weight $-1$ on the linear fibres of the projection $\pi\colon T_M\to M$.  
\end{lemma}

\begin{proof}
Take local co-ordinates $(z_i)$ on $M$ and $(z_i,\theta_j)$ on $X$ as above. The $\bC^*$-action on $M$ is given  by $t\cdot (z_i)= (tz_i)$. The derivative of the action map $a\colon \bC^*\times M\to M$ defines a $\bC^*$-action on $X$ given by $t\cdot (z_i,\theta_j)= (tz_i,t\theta_j)$. Composing with the weight $-1$ rescaling action on  the fibres gives the action $t\cdot (z_i,\theta_j)=(tz_i,\theta_j)$ whose generating vector field is $E$.
\end{proof}

\subsection{Joyce structures}
\label{joydef}

We define a Joyce structure by combining  a  pre-Joyce structure with a compatible period structure. Given a symplectic form $\omega\colon T_M\times T_M\to \O_M$ there is an induced pairing  $\eta\colon T_M^*\times T_M^*\to \O_M$  defined by the condition that the induced maps $\omega^{\flat}\colon T_M\to T_M^*$ and $\eta^\flat\colon T_M^*\to T_M$ are mutually inverse. We refer to $\eta$ as the inverse of $\omega$.




\begin{defn}
\label{joyce}
  A {Joyce structure} on a complex manifold $M$ consists of
\begin{itemize}
\item[(a)] a period structure $(T_M^{\bZ},Z,\nabla)$ on $M$,
\item[(b)] a  pre-Joyce structure $(\omega,h)$ on $M$, \end{itemize}
satisfying the following compatibility conditions:
\begin{itemize}
\item[(J1)]  if $\eta$ is the inverse of $\omega$ then $(2\pi i)^{-1} \eta$ takes integral values on the  dual lattice $(T_M^{\bZ})^*\subset T_M^*$,

\item[(J2)]  the connection $h$ is invariant under translations by the lattice $(2\pi i)\,T_M^{\bZ}\subset T_M$,

\item[(J3)] if $E$ is the $\nabla$-horizontal lift of the vector field $Z$, then for any vector field $u$ on $M$ \begin{equation}h([Z,u])=[E,h(u)],\end{equation}

\item[(J4)] the connection $h$ is invariant under the action of the involution $-1\colon X\to X$ which acts by multiplication by $-1$ on the fibres of $\pi\colon X=T_M\to M$.
\end{itemize}
We say that a Joyce structure is homogeneous if the underlying period structure is.
\end{defn}

Given a period structure on a complex manifold $M$ we  introduce the quotient
\begin{equation}
\label{covid}X^{\hash}=T_M^{\hash}=T_M/(2\pi i)\,T_M^{\bZ}.\end{equation}
Axiom (J2) is the statement that the connection $h$ descends to a connection on the $(\bC^*)^{n}$-bundle over $M$ given by the projection  $\pi\colon X^{\hash}\to M$.

\begin{remark}
    In the context of DT theory, in which $M$ is a space of numerical stability conditions on a CY$_3$ category, the  period structure $(T_M^{\bZ},Z,\nabla)$ and symplectic form $\omega$ are  well-known and immediate: the integral linear co-ordinates are given by central charges, and the symplectic form is induced by the inverse of the Euler form. The extra content of the Joyce structure is then just    the non-linear connection $h$, which  however is required to  satisfy a complicated system of   symmetry and curvature conditions. 
\end{remark}

\subsection{Associated \hk structures}
Given a Joyce structure on a complex manifold $M$,  the associated complex \hk structure  $(g,I,J,K)$ on $X=T_M$ defined in Section \ref{comphk} has certain extra symmetry properties.  For example, axiom (J2) ensures that  it is invariant under the translations by the lattice $(2\pi i)\,T_M^{\bZ}\subset T_M$, and hence descends to the quotient manifold $X^\hash$ introduced above. The other axioms give the following result. 

\begin{lemma}
There are identities

\begin{align}
\label{scales}  \Lie_E(I)&=0,  &\Lie_E(J\pm iK)&=\pm(J\pm iK),& \Lie_E(g)&=g.\\
\label{scales2} (-1)^*(I)&=I, & (-1)^*(J\pm iK)&=-(J\pm iK),& (-1)^*(g)&=-g.
\end{align}
\end{lemma}

\begin{proof}
We will evaluate both sides  of each  identity on vector fields of the form $h(u)$ and $v(u)$, where $u$ is an  arbitrary vector field  on $M$.
For   the identities \eqref{scales} note first that we have
\begin{equation}\label{foot}
\cL_E(h(u))=h(\cL_Z(u)), \qquad \cL_E(v(u))=v(\cL_Z(u)+u)
    \end{equation}
    The first of these equations is axiom (J3), and the second follows\todo{How?}  from Lemma \ref{jen}.

If $u$ is a vector field on $M$ then using  \eqref{here} we have

\begin{align}(\cL_E(J-iK)) \, h(u)&=\cL_E((J-iK) \, h(u))-(J-iK)\, \cL_E(h(u))=0,\\
\begin{split}(\cL_E(J-iK))\, v(u)&=\cL_E((J-iK)\, v(u))-(J-iK)\,\cL_E(v(u)) \\
    &=2h(\cL_Z(u))-2h(\cL_Z(u)+u)=-2h(u),\end{split}
\end{align}
and it follows that $\cL_E(J-iK)=-(J-iK)$. The claim $\cL_E(J+iK)=J+iK$ follows in the same way.
The statement that $\cL_E(I)=0$ holds because by \eqref{foot}
\begin{equation}(\cL_E(I))\, h(u)=\cL_E(I(h(u)))-I(\cL_E(h(u)))=\cL_E(i h(u))-i\cL_E(h(u))=0,\end{equation}
with an analogous identity for $v(u)$.

For the statement on the metric note first that (J1) implies that
    $\cL_Z(\omega)=2\omega$.
Then using the definition \eqref{g} of the metric $g$ we have

\begin{equation}\begin{split} \cL_E(g)(h(u_1),v(u_2))&=E \cdot g(h(u_1),v(u_2))-g(\cL_E (h(u_1)),v(u_2))-g(h(u_1),\cL_E (v(u_2)))\\&= E\cdot  \half\pi^*(\omega(u_1,u_2))-\half \pi^*(\omega(\cL_Z(u_1),u_2)) -\half\pi^*(\omega(u_1,\cL_Z(u_2)+u_2)) \\&=\half \pi^*(\cL_Z(\omega)(u_1,u_2))-\half \pi^*(\omega(u_1,u_2))=\half \pi^*(\omega(u_1,u_2))\\&=g(h(u_1),v(u_2)),\end{split}\end{equation}
where we used the fact that $\pi_*(E_x)=Z_{\pi(x)}$ for all $x\in X $ to write $E\cdot \pi^*(f)=\pi^*(Z\cdot f)$ for any  function $f$ on $M$.
An easier argument gives
\begin{equation}\cL_E(g)(h(u_1),h(u_2))=0=\cL_E(g)(v(u_1),v(u_2)),\end{equation}
so we find that $\cL_E (g) = g$ as required.

The  identities \eqref{scales2} follow from the definitions \eqref{ijk} and \eqref{g} together with
\begin{equation}  (-1)^* (h(u))=h(u), \qquad (-1)^* (v(u)) =-v(u).\end{equation}
The first of these equations is axiom (J4), and the second is immediate from the definition of the involution $-1$.
\end{proof}

\subsection{Pleba{\'n}ski function}
\label{plebe}
Consider a local system of integral linear co-ordinates $(z_1,\cdots,z_n)$ on $M$.  Taking associated co-ordinates $(z_i,\theta_j)$ on $X=T_M$ 
we can express the connections $h_\epsilon$  using locally-defined functions $W_i=W_i(z,\theta)$ as in Section \ref{prejoyce}. As before we fix the integration constants by imposing the condition    that $W_i$ vanishes on the zero section $M\subset X=T_M$.

\begin{lemma}
\label{lemmy}
   There is a unique locally-defined function $W=W(z,\theta)$ which vanishes on the zero section $M\subset X=T_M$ and satisfies 
$W_i=\frac{\partial W}{\partial \theta_i}$ for all $ 1\leq i\leq n$. It also satisfies the relations
   
   \begin{gather}
\label{point_intro}
\frac{\partial^2 W}{\partial \theta_i \partial z_j}-\frac{\partial^2 W}{\partial \theta_j \partial z_i }=\sum_{p,q} \eta_{pq} \cdot \frac{\partial^2 W}{\partial \theta_i \partial \theta_p} \cdot \frac{\partial^2 W}{\partial \theta_j \partial \theta_q}\\\label{gove}\frac{\partial^2 W}{\partial \theta_i \partial \theta_j}(z_1,\cdots,z_n,\theta_1+2\pi i k_1,\cdots,\theta_n+2\pi ik_n)=\frac{\partial^2 W}{\partial \theta_i \partial \theta_j}(z_1,\cdots,z_n,\theta_1,\cdots, \theta_n), \\
\label{reeves2}W(\lambda z_1,\cdots,\lambda z_n,\theta_1,\cdots, \theta_n)=\lambda^{-1} \cdot W(z_1,\cdots,z_n,\theta_1,\cdots, \theta_n),\\\label{starmer2}
W(z_1,\cdots,z_n,-\theta_1,\cdots, -\theta_n)=-W(z_1,\cdots,z_n,\theta_1,\cdots, \theta_n),\end{gather}
for all $1\leq i,j\leq n$, where $(k_1,\cdots, k_n)\in \bZ^n$ and $\lambda\in \bC^*$.
\end{lemma}

\begin{proof}
   Using the expressions \eqref{ze} the axioms (J2) -- (J4)  become the conditions

\begin{gather}
\label{gove2}
\frac{\partial W_i}{ \partial \theta_j}(z_1,\cdots,z_n,\theta_1+2\pi i k_1,\cdots,\theta_n+2\pi ik_n)=\frac{\partial W_i}{\partial \theta_j }(z_1,\cdots,z_n,\theta_1,\cdots, \theta_n),
\\
\label{reeves} \frac{\partial W_i}{\partial \theta_j}(\lambda z_1,\cdots,\lambda z_n,\theta_1,\cdots, \theta_n)=\lambda^{-1} \cdot \frac{\partial W_i}{\partial \theta_j}(z_1,\cdots,z_n,\theta_1,\cdots, \theta_n),\\
\label{starmer}
\frac{\partial W_i}{\partial \theta_j}(z_1,\cdots,z_n,-\theta_1,\cdots, -\theta_n)=-\frac{\partial W_i}{\partial \theta_j}(z_1,\cdots,z_n,\theta_1,\cdots, \theta_n).\end{gather}

Note that \eqref{starmer} implies that both sides vanish along the zero section $M\subset X=T_M$ where all $\theta_k=0$. It then follows  that for all $1\leq i,j\leq n$
\begin{equation}
\label{nick}
    \frac{\partial W_i}{\partial \theta_j}-\frac{\partial W_j}{\partial \theta_i}=0.
\end{equation}
Indeed, \eqref{southgate} shows that this expression is independent of the  co-ordinates $\theta_k$, and so if  it vanishes on the zero section it must be identically zero. It follows that there is a  locally-defined function $W=W(z,\theta)$ such that $W_i=\partial W/\partial \theta_i$. We again fix the integration constants by assuming that 
 $W$ vanishes on the zero section $M\subset X=T_M$. In view of \eqref{fix1} we therefore have
\begin{equation}
\label{fix}
    W(z_1,\cdots,z_n, 0, \cdots, 0)=0= \frac{\partial W}{\partial \theta_i}(z_1,\cdots, z_n,0,\cdots,0).\end{equation}
The equations \eqref{point_intro} then follow from \eqref{jude}. Indeed \eqref{jude} shows that the difference of the two sides is independent of the co-ordinates $\theta_k$, but by \eqref{starmer} and \eqref{fix} both sides vanish along the zero section.
Similarly, we deduce \eqref{reeves2} from  \eqref{reeves}, and \eqref{starmer2} from \eqref{starmer}.
\end{proof}

The function $W$ is called the Pleba{\'n}ski function, and the partial differential equations  \eqref{point_intro} are  known  as Pleba{\'n}ski's second heavenly equations.\todo{Comment about poles causing problems with integration constants.} 





\section{Twistor space}
\label{twist}

In this section we define the twistor space $p\colon Z\to \bP^1$ associated  to a pre-Joyce structure on a complex manifold $M$. It is defined as  the space of leaves of a foliation on $ \bP^1\times X$, where as before $X=T_M$ denotes the total space of the tangent bundle of $M$. The construction coincides with the Penrose twistor space construction \cite{Pen} applied to the complex \hk manifold on $X$.
We explain some special properties enjoyed by the twistor space of a Joyce structure, and derive an equation describing the twistor lines. In the final part we use this equation to revisit the  analogy between Joyce and Frobenius structures which was the main topic of \cite{RHDT2}.

\subsection{Construction of the twistor space}
\label{deft}

Let
$M$ be a complex manifold equipped with a pre-Joyce structure, and set $X=T_M$. For each point $(\epsilon_0,\epsilon_1)\in \bC^2\setminus\{0\}$  there is a bundle map \begin{equation}\epsilon_0 v+\epsilon_1 h\colon \pi^*(T_M)\to T_X,\end{equation} whose image, which depends only on the corresponding point $\epsilon=[\epsilon_0:\epsilon_1]\in \bP^1$, is a half-rank sub-bundle $H(\epsilon)\subset T_X$. The definition  of a pre-Joyce structure ensures that this sub-bundle  is closed under Lie bracket,
and hence defines a foliation on $X$. The twistor space $Z$ defined below will have a map $p\colon Z\to \bP^1$ whose fibre over the point $\epsilon\in \bP^1$ is the space of leaves of this foliation.

To give a global definition of the twistor space $Z$ we first recall that the tangent bundle of the product $\bP^1\times X$ has a canonical direct sum decomposition  \begin{equation}
    \label{decomp}
T_{\bP^1\times X}=\pi_1^*(T_{\bP^1})\oplus \pi_2^*(X),\end{equation} where $\pi_1,\pi_2$ denote the projections from $ \bP^1\times X$ onto the two factors. There is a half-rank sub-bundle $H\subset \pi_2^*(T_X)$ which when restricted to a fibre $\pi_1^{-1}(\epsilon)$ is the sub-bundle $H(\epsilon)\subset T_X$. Composing with the canonical inclusion $\pi_2^*(T_X)\subset T_{ \bP^1\times X}$ we can view $H$ as a sub-bundle of $T_{ \bP^1\times X}$ and this  is easily seen to be closed under Lie bracket. The twistor space $Z$ is then defined to be the space of leaves of the associated foliation on $\bP^1\times X$.

We denote by $q\colon \bP^1\times X\to Z$ the quotient map. There is an induced projection $p\colon Z\to \bP^1$ which satisfies $p\circ q =\pi_1$. We denote by $Z_\epsilon=p^{-1}(\epsilon)\subset Z$ the twistor fibre over $\epsilon\in \bP^1$, and $q_\epsilon\colon X\to Z_\epsilon$ the induced quotient map. For each $\epsilon\in \bP^1$ there is a commutative diagram
\begin{equation}
\label{keir}
\xymatrix@C=1.7em{  X \ar@{^{(}->}[rr] \ar[d]_{q_\epsilon} &&  \bP^1\times X \ar[d]^{q}\ar@/^2.5pc/[dd]^{\pi_1} \\
Z_\epsilon \ar@{^{(}->}[rr]\ar[d] && Z\ar[d]^{p} \\  \{\epsilon\} \ar@{^{(}->}[rr]&& \bP^1 }\end{equation}
in which the horizontal arrows are the obvious closed embeddings.

Each point $x\in X$ determines a section  of   the map $p\colon Z\to \bP^1$\begin{equation}
\label{twistorline}\sigma_x\colon \bP^1 \to Z, \qquad \epsilon\mapsto q(\epsilon,x),\end{equation}
which we refer to as a twistor line.

Recall the complex \hk structure $(g,I,J,K)$ on $X$ and the associated closed 2-forms $\Omega_{\pm}$ and $\Omega_I$ defined by \eqref{fart}. As in the proof of Theorem \ref{jon}, an easy calculation using \eqref{ijk} shows that 
\begin{equation}\label{soeasy}H(\epsilon)=\im\left(\epsilon_0v+\epsilon_1h\right)=\ker\left(\epsilon_0^2 (J+iK)+2i \epsilon_0\epsilon_1  I + \epsilon_1^2(J-iK)\right).\end{equation}
Since the right-hand side is precisely the kernel of the closed 2-form
\begin{equation}
    \label{lunch}
\epsilon_0^2 \, \Omega_+ +2i \epsilon_0\epsilon_1 \, \Omega_I + \epsilon_1^2 \,\Omega_-,\end{equation}
this 2-form descends along the map $q_\epsilon\colon X\to Z_\epsilon$, and defines a symplectic form $\Omega_\epsilon$ on $Z_\epsilon$. 

 In more global terms, the formula \eqref{lunch} defines a twisted relative   2-form on $\bP^1\times X$, namely a section of the bundle $\pi_1^*(\O(2))\tensor \pi_2^*(\wedge^2\, T_X^*)$. This descends along the quotient map $q\colon X\to Z$ to give a twisted  relative symplectic form  on the twistor space $p\colon Z\to \bP^1$. By definition, this is the section  $\Omega$ of the bundle $ p^*(\O(2))\tensor \wedge^2 \, T_{Z/\bP^1}^*$ uniquely defined by the condition
\begin{equation} \label{kalym} q^*(\Omega)=\epsilon_0^2\, \Omega_+ +2i \epsilon_0\epsilon_1\Omega_I + \epsilon_1^2\,\Omega_-.\end{equation}

From now on we will take the usual embedding $\bC\subset \bP^1$ with affine co-ordinate $\epsilon=\epsilon_1/\epsilon_0$, and write $\bP^1=\bC\cup\{\infty\}$ with $\infty=[0:1]$.  
When $\epsilon\in \bC\cup\{\infty\}$, the twistor fibre $Z_\epsilon$  is the space of leaves of the foliation on $X$ associated to the integrable distribution $\im(h_\epsilon)\subset T_X$,  where as before we write $h_\epsilon=h+\epsilon^{-1}v$. On the other hand the twistor fibre $Z_0$ is the space of leaves of the foliation on $X$ associated to  the vertical sub-bundle $\im(v)=\ker(\pi_*)$, and is therefore identified with $M$. 

Restricting $\Omega$ to a twistor fibre $Z_\epsilon$ gives a complex symplectic form $\Omega_\epsilon$, well-defined up to multiplication by a nonzero constant. When $\epsilon\in \bC^*\cup\{\infty\}$ we fix this scale by taking
\begin{equation}
\label{qe}q_\epsilon^*(\Omega_\epsilon)=\epsilon^{-2} \,\Omega_+ +2i\epsilon^{-1} \,\Omega_I+ \Omega_- ,\end{equation}
whereas for  $\epsilon=0$ we take  $\Omega_0=\omega$. We then have relations
\begin{equation}
\label{dumb}q_0^*(\Omega_0)=\Omega_+, \qquad q_\infty^*(\Omega_\infty)=\Omega_-.\end{equation}

\begin{remark}
Unfortunately, to obtain a well-behaved twistor space $p\colon Z\to \bP^1$  we cannot in general just  take the space of leaves of  the foliation  on $\bP^1\times X$. Rather, we  should consider the holonomy groupoid, which leads to the analytic analogue of a Deligne-Mumford stack \cite{Moer}.   We will completely ignore these subtleties here, and essentially pretend that $Z$ is a complex manifold. 
In fact, for what we do here,  nothing useful would be gained by a more abstract point of view, because  we are  only really using the twistor space as a useful and suggestive shorthand.  All statements we make about the space $Z$ can be easily translated into statements only involving    objects on $X$. For example, a symplectic form on the twistor fibre $Z_\epsilon$ is nothing but a closed 2-form on $X$ whose kernel coincides with the sub-bundle $H(\epsilon)\subset T_X$. Similarly,  an {\'e}tale map from the twistor fibre $Z_\epsilon$ to some complex manifold $Y$ is just a holomorphic map $f\colon X\to Y$ such that $\ker(f_*)=H(\epsilon)$.
\end{remark}


\subsection{Twistor space of a Joyce structure}

Suppose now that we have a Joyce structure on a complex manifold $M$ and consider the associated twistor space $p\colon Z\to \bP^1$.
Recall the vector field $Z$ on $M$ and the horizontal lift $E$ on $X$. 
Using the decomposition \eqref{decomp} we can define a vector field  on the product $ \bP^1\times X$ as the sum
\begin{equation}\tilde{E}=\epsilon\, \frac{d}{d\epsilon}+E.\end{equation}

\begin{lemma}
\label{election}
The vector field $\tilde{E}$ descends along the quotient map $q\colon  \bP^1\times X\to Z$. 
\end{lemma}

\begin{proof}
Let $u$ be a vector field on $M$. Writing $h_\epsilon(u)=h(u)+\epsilon^{-1}v(u)$ and using  \eqref{foot} we have
\begin{equation}\begin{split}\big[\tilde{E},h_\epsilon(u))\big]&=[E,h(u)]+\epsilon^{-1}[E,v(u)]-\epsilon^{-1}v(u)\\
&=h([Z,u])+\epsilon^{-1} v([Z,u])=h_\epsilon([Z,u]).\end{split}\end{equation}
Thus $\cL_{\tilde E}$ preserves the sub-bundle $H_\epsilon=\ker(q_\epsilon)_*$, and the claim follows.
\end{proof}

Let us specialise to the case of a homogeneous Joyce structure. The $\bC^*$-action on $M$ generated by the vector field $Z$ induces a $\bC^*$-action  on $X$ with generating vector field $E$ as in Lemma \ref{jen}. Combining this with the standard action of $\bC^*$  on $\bP^1$ rescaling $\epsilon$ with weight 1, we can then consider the diagonal action  on $X\times \bP^1$. It follows from Lemma \ref{election} that this  descends to an action on $Z$. Thus we obtain $\bC^*$-actions on each of the spaces in the right-hand column of \eqref{keir}, and the vertical maps $q$ and $p$ intertwine these actions.

Since  the map $p\colon Z\to \bP^1$ is $\bC^*$-equivariant, there are at most three distinct twistor fibres: $Z_0$, $Z_1$ and $Z_\infty$, and as above there is an identification  $Z_0=M$.\footnote{Note that, unlike in the case of real \hk manifolds, there is no requirement for an involution of the twistor space $Z$ lifting the antipodal map on $\bP^1$. In particular,  there need be no relation between the twistor fibres $Z_0$ and $Z_\infty$.}  
Using the definition \eqref{fart} and the identities \eqref{scales} we have
\begin{equation}\label{blbl}\cL_E(\Omega_+)=2\Omega_+, \qquad \cL_E(\Omega_I)=\Omega_I, \qquad \cL_E(\Omega_-)=0.\end{equation}The $\bC^*$-action on $Z$ induces a $\bC^*$-action on $Z_0$ and it follows from \eqref{blbl} that the symplectic form $\Omega_0$  is  homogeneous of weight 2 for this action.  Similarly there is an induced $\bC^*$-action on $Z_\infty$, but in this case the symplectic form  $\Omega_\infty$ is invariant.

We can use the $\bC^*$-action to trivialise the  map $p\colon Z\to \bP^1$ over the open subset $\bC^*\subset \bP^1$. We obtain a commutative diagram
\begin{equation}
\label{corona}
\begin{gathered}
\xymatrix@C=1.5em{  
 \bC^*\times Z_1 \ar[rr]^{m}\ar[d]_{\pi_1} && p^{-1}(\bC^*)\ar[d]^p \ar@{^{(}->}[rr] && Z \ar[d]^p\\
\bC^*\ar@{<->}[rr]^{=} && \bC^* \ar@{^{(}->}[rr] && \bP^1}\end{gathered}\end{equation}
where $\pi_1$ is the projection onto the first factor, and the isomorphism $m$ is  the restriction of the action  map $m\colon \bC^*\times Z\to Z$.

Consider the map $y\colon \bC^*\times X\to Z_1$ given by 
\begin{equation}\label{dr}y(\epsilon,x)=\epsilon^{-1}\cdot q(\epsilon,x)=q(1,\epsilon^{-1}\cdot x).\end{equation}
Note that, under the trivialisation \eqref{corona},  the restriction of the twistor line $\sigma_x\colon \bP^1\to Z$ becomes the section of $\pi_1$ given by $\epsilon\mapsto (\epsilon, y(\epsilon,x))$.
Let us take a system of integral linear co-ordinates $(z_1,\cdots,z_n)$ on $M$ and consider the induced co-ordinates $(z_i,\theta_j)$ on $X=T_M$ as in Section \ref{prejoyce}. 

 \begin{lemma}
 \label{men}
 The map $y(\epsilon,x)$ satisfies the equation
\begin{equation}\label{nonlinearguy}\frac{\partial }{\partial  \epsilon} \, y(\epsilon,x)= \bigg(\frac{1}{\epsilon^2} \cdot \sum_i z_i\cdot \frac{\partial }{\partial \theta_i} +\frac{1}{\epsilon} \cdot  \sum_{i,p,q} \eta_{pq}  \cdot z_i\cdot \frac{\partial^2 W}{\partial \theta_i\partial \theta_p} \cdot \frac{\partial }{\partial \theta_q}\bigg) \,y(\epsilon,x),\end{equation} 
  \end{lemma}
  
  \begin{proof}
The formula $y(\epsilon,x)=q(1,\epsilon^{-1}\cdot x)$ shows that $y$  is invariant under the diagonal $\bC^*$-action on $X\times \bC^*$, and is therefore  annhilated by the vector field \begin{equation}\tilde{E}=\sum_i z_i \cdot \partial/\partial z_i +\epsilon \cdot \partial/\partial \epsilon.\end{equation} 
 On the other hand, the  formula $y(\epsilon,x)=\epsilon^{-1}\cdot q(\epsilon,x)=\epsilon^{-1}\cdot q_\epsilon(x)$ shows that  $y$ factors via the map
 \begin{equation}1\times q_\epsilon\colon \bC^*\times X\to \bC^*\times Z_\epsilon.\end{equation} By definition of the twistor space it is therefore also  annhilated by the vector fields $h_\epsilon(\partial/\partial z_i)=h_i+\epsilon^{-1}v_i$. The result then follows from the formulae  \eqref{v} and \eqref{above}.
 \end{proof}

\subsection{Preferred co-ordinates and Stokes data}
\label{cancor}

The definition of a Joyce structure was first identified in \cite{RHDT2} by considering an analogy between the wall-crossing formula in DT theory and a class of iso-Stokes deformations familiar in the theory of Frobenius manifolds. This analogy can be re-expressed in terms of the geometry of the twistor space $p\colon Z\to \bP^1$, and leads to interesting conjectural properties of $Z$ which we attempt to explain here. This section is  speculative, and can be safely skipped. The basic point is that one should expect preferred systems of co-ordinates on the twistor fibre $Z_1$, in terms of which the twistor lines, viewed as maps $\bC^*\to Z_1$ via the trivialisation \eqref{corona}, have good asymptotic properties as $\epsilon\to 0$.

Let us fix a point $m\in M$ and restrict attention to  points $x\in T_{M,m}\subset X$.
 Changing the notation slightly, we obtain a family of maps $y(\epsilon)\colon T_{M,m}\to Z_1$. Thus for each point $x\in T_{M,m}$, the restriction of the twistor line $\sigma_x\colon \bP^1\to Z$  corresponds under the trivialisation \eqref{corona} to the map $\epsilon\mapsto (\epsilon,y(\epsilon)(x))$. It is then interesting to observe that the equation  \eqref{nonlinearguy} controlling the variation of $y(\epsilon)$   is formally analogous to the  linear equation
 \begin{equation}
 \label{linearguy}\frac{d}{d\epsilon} \,y(\epsilon)=\bigg(\frac{U}{\epsilon^2}+\frac{V}{\epsilon} \bigg)y(\epsilon)\end{equation}
 occurring in the theory of Frobenius manifolds. In the equation \eqref{linearguy} the matrices $U$ and $V$ are  infinitesimal linear automorphisms of the tangent space $T_m M$ at a point $m\in M$ of a Frobenius manifold, whereas the corresponding quantities in    \eqref{nonlinearguy} are infinitesimal automorphisms of the    vector space $T_{M,m}$ preserving  the linear symplectic form $\omega_m$.   This  analogy  was the main topic of \cite{RHDT2}.

Suppose that the matrix $U$  in \eqref{linearguy} has distinct eigenvalues $u_i\in \bC$. The Stokes rays of the equation \eqref{linearguy} are then defined to be the rays $\ell_{ij}=\bR_{>0} \cdot (u_i-u_j)\subset \bC^*$. A result of Balser, Jurkat and Lutz \cite{BJL} shows that in any half-plane $H(\varphi)\subset \bC^*$ centred on a non-Stokes ray $r=\bR_{>0}\cdot \exp({i\pi\varphi})\subset \bC^*$, there is a unique fundamental solution $\Phi\colon H(\varphi)\to \operatorname{GL}_n(\bC)$  to \eqref{linearguy} satisfying \begin{equation}
\label{7a}\Phi(\epsilon) \cdot \exp(U/\epsilon)\to \id \text { as } \epsilon\to 0.\end{equation}
Comparing these canonical solutions for half-planes $H(\varphi_\pm)$ centred on small  perturbations $r_\pm$ of a Stokes ray $\ell\subset \bC^*$ defines the Stokes factors $\bS(\ell)\in \operatorname{GL}_n(\bC)$.

Turning to the equation \eqref{nonlinearguy}, note that the first term on the right-hand-side is the translation-invariant vector field $U=\sum_i z_i \cdot \partial/\partial \theta_i$ on  $T_{M,m}$ corresponding to the value $Z_m\in T_{M,m}$  of the vector field $Z$  defined by the period structure on $M$.  Exponentiating this vector field yields well-defined linear automorphisms $\exp(U/\epsilon)$  of  the vector space $T_{M,m}$ given  in co-ordinates by $\theta_i\mapsto \theta_i+\epsilon^{-1} z_i$.
The above analogy then suggests that there should be a countable collection of Stokes rays $\ell\subset \bC^*$, and for any half-plane $H(\varphi)\subset \bC^*$ centred on a non-Stokes ray $r=\bR_{>0}\cdot \exp(i\pi\varphi)\subset \bC^*$, a symplectomorphism $F\colon Z_1\to T_{M,m}$, such that  $\Phi(\epsilon):=F\circ y(\epsilon)\colon T_{M,m}\to T_{M,m}$ satisfies\footnote{The  factors in \eqref{7a} and \eqref{7a+} appear in different orders because we are working in the group of symplectic automorphisms of the space $\bC^n$ rather than the opposite group of Poisson automorphisms of its ring  of functions: compare \cite[Section 6.6]{RHDT2}.}
\begin{equation}
\label{7a+}\exp(U/\epsilon)\circ \Phi(\epsilon) \to \id\text { as } \epsilon\to 0.\end{equation}
The maps $F$ corresponding to different half-planes $H(\varphi)$ then differ by compositions of symplectic automorphisms $\bS(\ell)\in \Aut_{\omega_p}(T_{M,m})$ associated to rays $\ell\subset \bC^*$. These automorphisms $\bS(\ell)$  should be viewed as non-linear Stokes factors.  

In practice  the above analogy should be considered a guiding principle rather than a precise statement. In particular we cannot expect the map $F$ to be defined on the whole of $Z_1$. Nonetheless the basic point is that once we have  chosen a system of integral linear co-ordinates $(z_1,\cdots,z_n)$ near the point  $m\in M$,  there should be preferred systems of Darboux co-ordinates  $(t_1,\cdots,t_{n})$\footnote{Unfortunately these co-ordinates are usually denoted $(x_1,\cdots,x_n)$ since in certain examples they are the logarithms of cluster $X$ co-ordinates. We can only apologise for this excess of $x$'s.} on open subsets of $Z_1$, depending on a choice of half-plane $H(\varphi)\subset \bC^*$. Given a  point $x\in X$ with local co-ordinates $(z_i,\theta_j)$ these should satisfy
 \begin{equation}
\label{asym}t_i(y(\epsilon)(x))\sim -\epsilon^{-1}z_i+\theta_i+O(\epsilon),\end{equation}
as $\epsilon\to 0$ in the half-plane $H(\varphi)$.

We observed in Section \ref{joydef} that axiom (J2) of Definition \ref{joyce}  implies that the connections $h_\epsilon=h+\epsilon^{-1}v$  descend to the quotient manifold  $X^\hash=T_M/(2\pi i)\, T_M^{\bZ}.$  
We can then define a twistor space $Z^\hash$ in exactly the same way as before by considering the space of leaves of the resulting foliation on $\bP^1\times X^\hash$. 
We can then repeat the above discussion replacing $Z_1$ with $Z^\hash_1$ and the vector spaces $T_{M,m}$ with the tori
\begin{equation}\label{torus}T_{M,m}^{\hash}=T_{M,m}/(2\pi i)\, T_{M,m}^{\bZ}\isom (\bC^*)^{n}.\end{equation}
This leads to an anaalogy between  Frobenius structures and Joyce structures  in which the group of linear automorphisms of the vector space $T_{M,m}$ is replaced by the group of symplectic automorphisms of the torus $T^\hash_{M,m}$.  

 \section{Further properties of twistor space}
 \label{letstwistagain}

Let $M$ be a complex manifold equipped with a Joyce structure. In this section we gather some miscellaneous results about the associated  twistor space $p\colon Z\to \bP^1$  associated to a Joyce structure. Firstly, we investigate Hamiltonian generating functions for the $\bC^*$-action on $Z_\infty$.  Secondly, we observe that the zero section $M\subset X=T_M$ is contracted by the quotient map $q_\infty \colon X\to Z_\infty$, and hence defines a distinguished point $0\in Z_\infty$. Structures on the tangent space $T_{0}\,  Z_\infty$ can then be pulled back to give linear structures on the  tangent bundle $T_M$. Finally we explain an alternative way to describe a Joyce structure in terms of local co-ordinates and a generating function. 
 
\subsection{Joyce function}
\label{joy}
From the Cartan formula and \eqref{blbl} we have
\begin{equation}di_E(\Omega_-)=\cL_E(\Omega_-)-i_E(d\Omega_-)=0.\end{equation}
Let us consider a  locally-defined function $F$ on $X$ satisfying $dF=-i_E(\Omega_-)$.  Then $F$ descends along the quotient map $q_\infty\colon X\to Z_\infty$. 
 Indeed, if a vector field $u$ on $X$ is horizontal for the connection $h=h_\infty$  then \begin{equation}u(F)=i_u(dF)=-i_u i_E(\Omega_-)=i_E i_u(\Omega_-)=0,\end{equation}
since $\Omega_-=q_\infty^*(\Omega_\infty)$.
We call the resulting  function $F\colon Z_\infty \to \bC$, or its pullback to $X$, a  Joyce function.\footnote{In \cite{RHDT2} we used the term Joyce function as a synonym for the Pleba{\'n}ski function $W$ of Section \ref{plebe}. Following \cite{AP} we now prefer to use it for the function introduced here, which was also considered by Joyce \cite{holgen}.} In the case of a homogeneous Joyce structure, $F$  is  a Hamiltonian generating function for the symplectic action of $\bC^*$ on $Z_\infty$.

Note that $F$ is only well-defined up to the addition of a constant. It follows from Lemma \ref{meterol} below that the pullback $q_\infty^*(F)$ is constant along the zero section $M\subset T_M=X$. We can therefore normalise $F$  by insisting that it vanishes on this locus. Assuming that $M$ is connected this is equivalent to the statement that $F$ vanishes on the distinguished point $0\in Z_\infty$ of Lemma \ref{meterol}.  With this normalisation $q_\infty^*(F)$ becomes a global function on $X$. 

To write an explicit expression for the Joyce function let us choose a local system of co-ordinates on $M$ as in Section \ref{plebe}, and denote by $W=W(z,\theta)$ the resulting Pleba{\'n}ski function.  

\begin{lemma}
\label{joy2}
The Joyce function $F$ is given by the expression
\begin{equation}\label{domi}F(z_i,\theta_j)=v(E)(W)=\sum_q z_q \cdot \frac{\partial W}{\partial \theta_q}\end{equation}
\end{lemma}

\begin{proof}
Using \eqref{w} and \eqref{point_intro}, and recalling that $W_i=\partial W/\partial \theta_i$, we have 

\begin{equation}\begin{split}i_E(\Omega_-)&=-\sum_{p,q} z_q\, \frac{\partial^2 W}{\partial \theta_p \partial\theta_q} \, d \theta_p  +\sum_{p,q} \frac{\partial^2 W}{\partial z_p \partial \theta_q} (z_p\, dz_q - z_q\, dz_p)\\
&=-\sum_{p,q} z_q\, \frac{\partial^2 W}{\partial \theta_p \partial\theta_q} \, d \theta_p- \sum_{p,q} z_q \, \frac{\partial^2 W}{\partial z_p \partial \theta_q} \, dz_p - \sum_q \frac{\partial W}{\partial  \theta_q}\, dz_q\\
&=-\sum_p \Big(\frac{\partial F}{\partial \theta_p} \, d\theta_p +  \frac{\partial F}{\partial z_p} \, dz_p\Big)=-dF,\end{split}\end{equation}
where we used the homogeneity property \eqref{reeves2}  in the form $\sum_p z_p \cdot \frac{\partial W}{\partial z_p} = - W$. The expression \eqref{domi} vanishes along the zero section $M\subset T_M=X$ by the relation \eqref{fix}.
\end{proof}


 \subsection{Distinguished point  of $Z_\infty$}
 \label{lin}
 In this section  we will assume that the base $M$ of our Joyce structure is connected. For simplicity we also assume that the Joyce structure is homogeneous, although this is not strictly necessary.
 
  \begin{lemma}
  \label{meterol}
The map $q_\infty \colon X\to Z_\infty$ contracts the  zero section $M\subset X=T_M$ to a single point $0\in Z_\infty$ which is a fixed point for the induced $\bC^*$-action on $Z_\infty$. 
 \end{lemma}
 
 \begin{proof}
Note that the parity property \eqref{starmer2}  and the formula \eqref{above} imply that along the zero section $M\subset X=T_M$ we have $h_i=\partial/\partial z_i$. The first claim follows immediately from this. The second claim holds because the  action of $\bC^*$  on $X$ preserves the zero section.
\end{proof}

 The operator $J\colon T_X\to T_X$ maps vertical tangent vectors to horizontal ones, and hence identifies the normal bundle to the zero section $M\subset T_M$ with the tangent bundle $T_M$. The derivative of the quotient map $q_\infty\colon X\to Z_\infty$ identifies this normal bundle with the trivial bundle with fibre $T_{Z_\infty,0}$. The combination of these two maps gives an isomorphism
 \begin{equation}
 \label{conne}T_{M,p}\lRa{J} N_{M\subset X,p} \lRa{q_{\infty,*}} T_{Z_\infty,0},\end{equation}
for each point $p\in M\subset T_M$, and hence a flat connection on the tangent bundle $T_M$. This is the linear Joyce connection from \cite[Section 7]{RHDT2}, which  appeared in the original paper of Joyce \cite{holgen}. In co-ordinates it is given by the formula\todo{Explain}
\begin{equation}
\label{st}\nabla^J_{\frac{\partial}{\partial z_i}}\Big(\frac{\partial}{\partial z_j}\Big)=  -\sum_{l,m}\eta_{lm}\cdot\frac{\partial^3 W}{\partial \theta_i \, \partial \theta_j \, \partial \theta_l}\Big|_{\theta=0} \cdot \frac{\partial}{\partial z_m}.\end{equation}
It was shown in \cite[Section 3.2]{Strachan} that $\nabla^J$  coincides with the Levi-Civita connection of the complex \hk structure on $X$ restricted to the zero-section $M\subset X=T_M$.

The following result follows immediately from the definitions.

\begin{lemma}
The weight space decomposition for the action of $\bC^*$ on   $T_{Z_\infty,0}$ defines via the identification \eqref{conne}  a decomposition
$T_M \isom \bigoplus_{i\in \bZ} V_i$ into $\nabla^J$-flat  sub-bundles  $V_i\subset T_M$.\qed
\end{lemma}
 
If the distinguished point $0\in Z_\infty$ is an isolated fixed point for the $\bC^*$ action, there are some additional consequences described in the following result. An example when this condition holds is  described in Section \ref{a2} below. 

 \begin{lemma}
 \label{lem}
 Suppose the point $0\in Z_\infty$ is an isolated fixed point for the action of $\bC^*$. 
 \begin{itemize}
 \item[(i)] The Hessian of the Joyce function $F$  defines a non-degenerate symmetric bilinear form on  $T_{Z_\infty,0}$. Via the identification \eqref{conne} this induces a metric on $M$   whose Levi-Civita connection is the linear Joyce connection $\nabla^J$. 

\item[(ii)]The positive and negative weight spaces of the $\bC^*$-action on $T_{Z_\infty,0}$ define via the identification \eqref{conne} a decomposition $T_M=V_-\oplus V_+$ into $\nabla^J$-flat sub-bundles.
 These are
 Lagrangian  for the symplectic form $\omega$.
\end{itemize}
\end{lemma}
 
 \begin{proof}
 Part (i) is immediate from the result of Lemma \ref{joy2}  that $F$ is the Hamiltonian for the $\bC^*$-action on $Z_\infty$. For part (ii), note that since $\Omega_-$ is $\bC^*$-invariant, the positive and negative weight spaces in $T_{Z_\infty,0}$ are Lagrangian for the form $\Omega_-$. The result then follows by noting that the operator $J$ exchanges the forms $\Omega_-=q_\infty^*(\omega_\infty)$ and $\Omega_+=\pi^*(\omega)$, so the composite \eqref{conne} takes the form $\omega_p$ to $\Omega_{\infty,0}$. 
 \end{proof}
 
The metric $g$ of Lemma \ref{lem}  is given in co-ordinates by the formula
\begin{equation}g\Big(\frac{\partial}{\partial z_i}, \frac{\partial}{\partial z_j}\Big)=\frac{\partial ^2 F}{\partial \theta_i \partial \theta_j}\Big|_{\theta=0}.\end{equation}
This is the Joyce metric of \cite[Section 7]{RHDT2}, which  also appeared in the original paper \cite{holgen}.

\subsection{Another Pleba{\'n}ski function}\label{u}
In Section \ref{plebe}  we described a Joyce structure in local co-ordinates using the Pleba{\'n}ski function $W=W(z,\theta)$. For the sake of completeness we briefly discuss here an alternative generating function, also introduced by Pleba{\'n}ski \cite{P}.  In the literature (see e.g. \cite {DM}) the function  $W$ is called the Pleba{\'n}ski function of the second kind, whereas the function $U$  introduced below is the Pleba{\'n}ski function of the first kind. 

Consider a Joyce structure on a complex manifold $M$.  We can  introduce
local co-ordinates on $X=T_M$  by combining local Darboux co-ordinates $(z_1,\cdots,z_n)$ on the twistor fibre $Z_0=M$ as in Section \ref{prejoyce}, with the pullback of local  Darboux co-ordinates $(\phi_1,\cdots,\phi_n)$ on the twistor fibre  $Z_\infty$. Then by definition
\begin{equation} \Omega_+=\frac{1}{2}\cdot\sum_{p,q} \omega_{pq}  \cdot dz_p \wedge dz_q,\qquad \Omega_-=\frac{1}{2}\cdot \sum_{p,q} \omega_{pq}  \cdot d\phi_p \wedge d\phi_q.\end{equation}

To find an expression  for $\Omega_I$ in the co-ordinate system $(z_i,\phi_j)$, note that since $Z_\infty$ is the quotient of $X$ by the distribution spanned by the vector fields $h_i$ of \eqref{above}, we can write
\begin{equation}
\label{above4}\frac{\partial}{\partial z_i}\bigg|_{\phi}=\frac{\partial}{\partial z_i}\bigg|_{\theta}+\sum_{p,q} \eta_{pq}\cdot \frac{\partial^2 W}{\partial \theta_i \partial \theta_p}\cdot \frac{\partial}{\partial \theta_q},\end{equation}
where the subscripts indicate which variables are being held fixed. Then
\begin{equation}\frac{\partial}{\partial z_r}\bigg|_{\phi} \Big( \sum_k \omega_{ks}\theta_k\Big)=\frac{\partial^2 W}{\partial \theta_r \partial \theta_s},\end{equation}
and the symmetry of the right-hand side ensures that there is  a locally-defined function $U=U(z_i,\phi_j)$ on $X$ satisfying the equations
\begin{equation}\label{id}\frac{\partial U}{\partial z_s}=\sum_k \omega_{ks}\theta_k, \qquad \frac{\partial ^2 U}{\partial z_r \partial z_s}=\frac{\partial^2 W}{\partial \theta_r\partial \theta_s}.\end{equation}
The formula \eqref{south} then becomes \begin{equation}
\label{nontrev} 2i \Omega_I=-d\Big(\sum_{p,q} \omega_{pq}  \theta_p \, dz_q\Big)=-\sum_{p,q} \frac{\partial^2 U}{\partial \phi_p \partial z_q} \cdot d\phi _p \wedge d z_q.\end{equation}

After restricting to a fibre of  $\pi\colon X\to M$ the formula \eqref{w} shows that 
\begin{equation}\sum_{p,q} \omega_{pq}  \cdot d\theta_p \wedge d\theta_q= \sum_{p,q} \omega_{pq}  \cdot d\phi_p \wedge d\phi_q,\end{equation}
which implies that 
\begin{equation}\sum_{p,q} \omega_{pq} \cdot \frac{\partial \theta_p}{\partial \phi_i} \cdot \frac{\partial \theta_q}{\partial \phi_j}=\omega_{ij}.\end{equation}
 Using \eqref{id} we then obtain the relations
\begin{equation}\label{firstpleb}\sum_{r,s} \eta_{rs} \cdot \frac{\partial^2 U}{\partial \phi_i \partial z_r} \cdot \frac{\partial^2 U}{\partial \phi_j \partial z_s} =-\omega_{ij}\end{equation}
which are  known as Pleba{\'n}ski's first heavenly equations.


\section{Hamiltonian systems}
\label{ham}

In this section we show how to use a Joyce structure to define a time-dependent Hamiltonian system. The construction depends on two additional pieces of data: a cotangent bundle structure on $Z_0$, and a Lagrangian submanifold $R\subset Z_\infty$. We give some motivation   in the next section, where, in the particular case of Joyce structures of class $S[A_1]$, we relate our Hamiltonian systems to isomonodromy connections.

\subsection{Hamiltonian systems}
\label{hamintro}

We begin with the following definition.

\begin{definition}
\label{dolly}
A time-dependent Hamiltonian system consists of the following data:
\begin{itemize}
\item[(i)] a submersion $f\colon Y\to B$ with a relative symplectic form $\Omega\in H^0(Y,\wedge^2\, T^*_{Y/B})$,
\item [(ii)] a flat, symplectic connection $k$ on $f$,
\item[(iii)] a section $\varpi\in H^0(Y, f^*(T_B^*))$.
\end{itemize}
\end{definition}

For each vector field  $u\in H^0(B,T_B)$ there is an associated function \begin{equation}H_u=(f^*(u),\varpi)\colon Y\to \bC.\end{equation} 
There is then a pencil $k_\epsilon$ of symplectic connections on $f$ defined by
\begin{equation}\label{hamil}k_\epsilon(u)=k(u)+\epsilon^{-1}  \cdot \Omega^\sharp(dH_u).\end{equation}
The system is called strongly-integrable if these connections are all flat.

These definitions become more familiar when expressed in local co-ordinates. Take co-ordinates $t_i$ on the base $B$, which we can think of as times, and $k$-flat Darboux co-ordinates $(q_i,p_i)$ on the fibres of $f$, so that $\Omega=\sum_i dq_i\wedge dp_i$. We can then write $\varpi=\sum_i H_i \, dt_i$ and view the functions $H_i\colon Y\to \bC$ as  time-dependent Hamiltonians. The connection $k_\epsilon$ is then given by the flows\begin{equation}\label{noteat}k_\epsilon\Big(\frac{\partial}{\partial t_i}\Big)=\frac{\partial}{\partial t_i}+\frac{1}{\epsilon} \cdot \sum_j\Big(\frac{\partial H_i}{\partial p_j}\frac{\partial }{\partial q_j}-\frac{\partial H_i}{\partial q_j}\frac{\partial }{\partial p_j}\Big).\end{equation} 
The condition that the system is strongly-integrable is that for all $1\leq i,j\leq n$
\begin{equation}\label{wan}\sum_{r,s}\Big(\frac{\partial H_i}{\partial q_r}\cdot \frac{\partial H_j}{\partial p_s}-\frac{\partial H_i}{\partial q_s}\cdot \frac{\partial H_j}{\partial p_r}\Big)=0, \qquad \frac{\partial H_i}{\partial t_j}=\frac{\partial H_j}{\partial t_i}.\end{equation}

For a nice exposition of Definition \ref{dolly} see  \cite[Section 5]{Boalch}. Note that Boalch works in the  real $C^\infty$ setting, whereas we assume, as elsewhere in the paper, that all structures are holomorphic.  Note  also that Boalch assumes that  $Y=M\times B$ is a global product, with $M$ a fixed symplectic manifold, and $k$  the canonical connection on the projection $f\colon M\times B \to B$. We can always reduce to this case by passing to an open subset of $Y$. 


\subsection{Hamiltonian systems from Joyce structures}
\label{hamham}

Let $M$ be a complex manifold with a holomorphic symplectic form $\omega$.  By a cotangent bundle structure on $M$ we mean the data of a complex manifold $B$ and an open embedding $M\subset T_B^*$, such that $\omega$ is the restriction of the canonical symplectic form on $T_B^*$. We denote by $\rho\colon M\to B$ the induced projection map.

The Liouville 1-form on $T^*B$ restricts to a 1-form  $\lambda\in H^0(M,T^*_M)$ on the open subset $M$. It will be convenient to distinguish $\lambda$ from the tautological section $\beta\in H^0(M,\rho^* (T_B^*))$. The two are identified via the inclusion  $\rho^*(T_B^*)\hookrightarrow T_M^*$ induced by $\rho$. 

Given local co-ordinates $(t_1,\cdots,t_d)$ on $B$, there are induced linear co-ordinates $(s_1,\cdots, s_d)$ on the cotangent spaces $T_{B,b}^*$ obtained by writing a 1-form as $\sum_i s_i\, dt_i$. In the resulting co-ordinates $(s_i,t_j)$  on $M$ we have \begin{equation}\omega=\sum_i dt_i\wedge ds_i, \qquad \beta=\sum_i s_i \cdot \rho^*(dt_i), \qquad \lambda=\sum_i s_i\, dt_i.\end{equation} 

Let $M$ be a  complex manifold equipped with a Joyce structure.  Thus
there is a pencil of flat, symplectic  connections $h_\epsilon=h+\epsilon^{-1} v$ on the projection  $\pi\colon X=T_M\to M$, and  closed 2-forms $\Omega_I$ and $\Omega_\pm$ on $X$. We denote by $p\colon Z\to \bP^1$ the associated twistor space. Suppose
 also given:
\begin{itemize}
\item[(i)] a cotangent bundle structure $M\subset T^*_B$,

\item[(ii)] a Lagrangian submanifold  $R\subset Z_\infty$.
\end{itemize}

Set $Y=q_\infty^{-1}(R)\subset X$, and denote by $i\colon Y\hookrightarrow X$  the inclusion.  There are maps
\begin{equation}\xymatrix@C=1em{ Y\,\ar@{^{(}->}[rr]^{i} && X  \ar[rr]^{\pi} && M \ar[rr]^{\rho} && B} \end{equation}
Define  $p\colon Y\to M$ and $f\colon Y\to B$ as the composites $p=\pi\circ i$ and $f=\rho\circ \pi\circ i$.
We make the  transversality assumption:
\begin{itemize}
\item[$(\star$)] For  each $b\in B$ the restriction of $q_1\colon X\to Z_1$ to the fibre $f^{-1}(b)\subset Y\subset X$ is {\'e}tale.\end{itemize}

The following result will be proved in the next section.

\begin{thm}
\label{inta}
Given the above data there is a strongly-integrable time-dependent Hamiltonian system on the map $f\colon Y\to B$ uniquely specified by the following conditions:
\begin{itemize}
 \item[(i)]  the relative symplectic form $\Omega$ is induced by the closed 2-form $i^*(2i\Omega_I)$ on $Y$;
 \item[(ii)] for each $\epsilon\in \bC^*$ the  connection $k_\epsilon$  on $f\colon Y\to B$ satisfies \begin{equation}
 \label{leyla}\im (k_\epsilon)=T_Y\cap \im (h_\epsilon)\subset T_X;\end{equation}
 \item[(iii)] the Hamiltonian form is  $\varpi=p^*(\beta)\in  H^0(Y,f^*(T_B^*))$.   \end{itemize}
\end{thm}

To make condition (iii) more explicit, take a co-ordinates system $(t_1,\cdots,t_d)$ on $B$, and extend to  a co-ordinate system   $(s_i,t_j)$ on $M$ as above. We can also extend to a   co-ordinate system $(t_i,q_j,p_j)$ on $Y$ as in Section \ref{hamintro}. Then  writing $H_i=p^*(s_i)$, we have $\varpi=\sum_i H_i \cdot  f^*(dt_i)$, and 
\begin{equation}k_\epsilon\Big(\frac{\partial}{\partial t_i}\Big)=\frac{\partial}{\partial t_i}+\frac{1}{\epsilon}\cdot  \sum_j \bigg(\frac{\partial H_i}{\partial q_j} \frac{\partial}{\partial p_j} - \frac{\partial H_i}{\partial p_j} \frac{\partial}{\partial q_j}\bigg).\end{equation}
The main claim of Theorem \ref{inta} is then  that  the connections $k_\epsilon$ are all flat, so that the conditions \eqref{wan} hold for the Hamiltonians $H_i$.

\subsection{Proof of Theorem \ref{inta}}

Take a point $y\in Y\subset X$ with $\pi(y)=m\in M$ and set $b=\rho(m)$. Given  a tangent vector $u\in T_{B,b}$ let us choose a lift  $w\in T_{M,m}$ satisfying $\rho_*(w)=u$. Using the maps $h_\epsilon,v\colon \pi^*(T_M)\to T_X$ as in Section \ref{prejoyce} we then obtain tangent vectors $h_\epsilon(w),v(w)\in T_{X,y}$. Recall that $h_\epsilon(w)=\epsilon^{-1}v(w)+h(w)$, and note that $h(w)\in T_{Y,y}$, since $Y=q_\infty^{-1}(R)$, and $q_\infty$ contracts the leaves of $h=h_\infty$. This implies that for any $\epsilon\in \bC^*$  the  following two conditions are equivalent:\begin{itemize}
\item[(a)] $v(w)\in T_{Y,y}\subset T_{X,y}$,
\item[(b)] $h_\epsilon(w)\in T_{Y,y}\subset T_{X,y}$.
\end{itemize}
Thus when (b) holds for some $\epsilon\in \bC^*$ it holds for all such $\epsilon$.

Consider next the transversality statement
\begin{itemize}
\item[$(\star)_\epsilon$] The restriction of $q_\epsilon\colon X\to Z_\epsilon$ to the fibre $f^{-1}(b)\subset Y\subset X$ is {\'e}tale at the point $y\in Y$.
\end{itemize}
Since the kernel of the derivative of $q_\epsilon$ at the point $y\in X$ is the image of $h_\epsilon(T_{M,m})$, the condition  $(\star)_\epsilon$ is equivalent to \begin{equation}h_\epsilon^{-1}(T_{Y,y})\cap \ker(\rho_*)=(0)\subset T_{M,m}.\end{equation} But this is 
also the condition for the existence of a unique lift $w\in T_{M,m}$ satisfying condition (b).  Thus since we assumed $(\star)_\epsilon$ for $\epsilon=1$, it holds for all $\epsilon\in \bC^*$.

We can now construct connections $k_\epsilon$ on $\pi\colon Y\to B$ for all $\epsilon\in \bC^*$ by setting $k_\epsilon(u)=h_\epsilon(w)$, where $w$ is the unique lift satisfying condition (b).
In more geometric terms,  the  condition $(\star)_\epsilon$ ensures that the map $(f,q_\epsilon)\colon Y\to B\times Z_\epsilon$ is  {\'e}tale, and the connection $k_\epsilon$ is then pulled back from the trivial connection on the projection $B\times Z_\epsilon\to B$.  This second description shows in  particular that the resulting connection $k_\epsilon$ is flat. 

The closed 2-form $q_\epsilon^*(\Omega_\epsilon)$ defines a relative symplectic form on $f$  by the condition $(\star_\epsilon)$. Recall the identity of closed 2-forms
\begin{equation}
\label{qe2}q_\epsilon^*(\Omega_\epsilon)=\epsilon^{-2} q_0^*(\Omega_0) + 2i\epsilon^{-1} \Omega_I+ q_\infty^*(\Omega_\infty )\end{equation}
from Section \ref{twist}. On restricting to $Y$ the last term on the right-hand side vanishes, since $R\subset Z_\infty$ is Lagrangian. On further restricting to a fibre $f^{-1}(b)\subset Y$ the first term also vanishes, since $\rho^{-1}(b)\subset M$   is Lagrangian. Thus for $\epsilon \in \bC^*$ the forms $\epsilon\cdot q_\epsilon^*(\Omega_\epsilon)$  define the same relative symplectic form $\Omega$ on $f$, and this is also induced by $2i\Omega_I$. Note that the the kernel of the restriction $\Omega_\epsilon|_Y$ clearly contains the subspace $\im(k_\epsilon)$, and hence coincides with it. This implies that the connection $k_\epsilon$ on $f$ is symplectic  \cite[Theorem 4]{Gotay}. 

Observe next that there is a commutative diagram 
\begin{equation}
\label{commie}
\xymatrix@C=2em{ 
 T_X\ar@{->}[rr]^{-(2i\Omega_I)^\flat} && T_X^*
 \\
 T_M\ar[rr]^{\omega^\flat}\ar[u]^{v} && T_M^* \ar[u]_{\pi^*}
} \end{equation}
since the relations \eqref{ijk} show that for tangent vectors $w_1$ to $M$ and $w_2$ to $X$  

\begin{equation}\begin{split}-2i\Omega_I(v(w_1),w_2)&=-g(2i I (v(w_1)),w_2)= -2g(v(w_1),w_2)= g((J+iK)(h(w_1)),w_2)\\
&=\Omega_+(h(w_1),w_2)= (\pi^*\omega)(h(w_1),w_2)=\omega(w_1,\pi_*(w_2)).\end{split}\end{equation}

The fact  that the closed 2-form $i^*(2i\Omega_I)$ induces a relative symplectic form on the map $f\colon Y\to B$ is the statement that the composite of  the bundle maps
\begin{equation}\xymatrix@C=1.6em{ \ker(f_*)\, \ar@{^{(}->}[rr] && T_Y  \ar[rr]^{(2i\Omega_I)^\flat} && T_Y^* \ar[rr] && T_Y^* /f^*(T_B^*)} \end{equation}
is an isomorphism.
The relative Hamiltonian flow $r$  corresponding to a function $H\colon Y\to \bC$ is then the unique  vertical vector field  on the map $f\colon Y\to B$ which  is mapped to $dH$ viewed as a section of $T_Y^*/ f^*(T_B^*)$. In symbols we can write $(2i\Omega_I)^\flat(r)=dH+f^*(\alpha)$ for some covector field $\alpha$ on $B$. 

Consider the canonical section $\varpi=p^*(\beta)\in H^0(Y,f^*(T^*_B))$. Note that given a vector field $u\in H^0(B,T_B)$ the corresponding Hamiltonian $H_u=(f^*(u),\varpi)=p^*(\rho^*(u),\beta)$ on $Y$ is pulled back from $M$. Let us lift $u$ to a vector field $w$  on $M$ as above. Then $k_\epsilon(u)=h_\epsilon(w)=h(w)+\epsilon^{-1}\cdot v(w)$. We claim that $v(w)$ is the relative Hamiltonian flow $r$ defined by the function $H_u$. To prove this we must show that the 1-form $(2i\Omega_I)^\flat(v(w))-dH_u$ on $Y$ is a pullback from $B$. By the commutative diagram \eqref{commie} this is equivalent to showing that the 1-form $\omega^\flat(w)-d(\rho^*(u),\beta)$ on $M$ is a pullback from $B$. 

This final step is perhaps most easily done in local co-ordinates $(s_i,t_i)$ on $M$ as above in which $\beta =\sum_i s_i dt_i$ and $\omega=\sum_i dt_i\wedge ds_i$.  If we take $u=\frac{\partial}{\partial t_i}$ then $(\rho^*(u),\beta)= s_i$ and the lift $w$ has the form $w=\frac{\partial}{\partial t_i}+\sum_j a_{j} \frac{\partial}{\partial s_j}$ for locally-defined functions $a_j\colon M\to \bC$. But then $\omega^\flat(w)=ds_i-\sum_j a_j \, dt_j$, and since the $t_i$ are pulled back via $\rho$ this proves the claim.
We have now defined the symplectic connections $k_\epsilon$ for $\epsilon\in \bC^*$ and proved the relation \eqref{noteat}. We can then define a symplectic connection $k=k_\infty$ by the same relation. Since the $k_\epsilon$ are flat for all $\epsilon\in \bC^*$ the relations \eqref{wan} hold, and it follows that $k$ is also flat.


\section{Joyce structures of class $S[A_1]$}
\label{geometric}

In this section we discuss an interesting class of examples of Joyce structures. They  are related to supersymmetric gauge theories of class $S[A_1]$, and were first constructed in the paper \cite{CH}. The base $M$  parameterises pairs $(C,Q)$ consisting of an algebraic curve $C$ of some fixed genus $g\geq 2$, and a quadratic differential $Q\in H^0(C,\omega_C^{\tensor 2})$ with simple zeroes.  The generalisation to the case of meromorphic quadratic differential with poles of fixed orders  will be treated in the forthcoming work \cite{Z}. After reviewing the construction of these Joyce structures, we relate the twistor fibre $Z_1$ to the associated character variety, and the Hamiltonian systems of Section \ref{ham} to isomonodromy connections. 

\subsection{Construction}
\label{con}

 For each point $(C,Q)\in M$ there is a branched double cover  $p\colon \Sigma\to C$ defined via the equation $y^2=Q(x)$, and equipped with a covering involution $\sigma\colon \Sigma\to \Sigma$. Taking periods of the form $y \, dx$ on $\Sigma$ identifies the tangent space $T_{(C,Q)} M$ with the anti-invariant cohomology group $H^1(\Sigma,\bC)^-$. The intersection pairing on $H_1(\Sigma,\bC)^-$ then induces a symplectic form on $M$, and the  dual of the integral homology groups $H_1(\Sigma,\bZ)^-$  defines an integral affine structure $T_M^{\bZ}\subset T_M$. There is a natural $\bC^*$-action on $M$ which rescales the quadratic differential $Q$ with weight 2. Taken together this defines a period structure  on $M$.
  
 The usual spectral correspondence associates to a $\sigma$-anti-invariant line bundle $L$ on $\Sigma$, a rank 2 vector bundle $E=p_*(L)\tensor \sqrt{\omega_C}\,$ on $C$ with a Higgs field $\Phi$. A key ingredient in \cite{CH} is an extension of this correspondence which relates anti-invariant connections $\partial$ on $L$ to  connections $\nabla$ on $E$.  Given this, we can view the space $X^\hash$ appearing in \eqref{covid} as parameterising the data $(C,E,\nabla,\Phi)$. The pencil of non-linear connections $h_\epsilon$ is then obtained by requiring that the monodromy of the connection $\nabla-\epsilon^{-1}\Phi$ is constant as the pair $(C,Q)$ varies.
 
To explain the construction in a little more detail, let us fix a parameter $\epsilon\in \bC^*$ and contemplate the following diagram of moduli spaces: 

\begin{equation}
\begin{gathered}\label{biggy}
\xymatrix@C=1.8em{  & \cM(C,E,\nabla,\Phi) \ar[dl]_{\alpha} \ar[dr]^{\beta_\epsilon}\\
\cM(C,Q,L,\partial)  \ar[d]_{\pi_3} && \cM(C,Q,E,\nabla_\epsilon)\ar[d]_{\pi_2} \ar@{->}[rr]^{\rho'} &&\cM(C,E,\nabla_\epsilon)\ar[d]_{\pi_1}   \\
\cM(C,Q) \ar@{<->}[rr]^{=} && \cM(C,Q)\ar@{->}[rr]^{\rho}&& \cM(C)
}\end{gathered}
 \end{equation}
Each moduli space  parameterises the indicated objects,  and the  maps $\rho, \rho'$ and $\pi_i$ are the obvious projections. The map $\alpha$ is the above-mentioned extension of the spectral  correspondence, and the map $\beta_\epsilon$ is defined by the rule
\begin{equation}\beta_\epsilon(C,E,\nabla,\Phi)=(C,-\det(\Phi), E, \nabla-\epsilon^{-1}\Phi).\end{equation}
An important  point is that $\alpha$ is birational, and $\beta_\epsilon$ is generically {\'e}tale.

Given a point of $M=\cM(C,Q)$, an anti-invariant  line bundle with connection $(L,\partial)$ on the spectral curve $\Sigma$ has an associated  holonomy representation $H_1(\Sigma,\bZ)^-\to \bC^*$. This determines $(L,\partial)$ up to an action of the group of 2-torsion line bundles on $C$. We therefore obtain an {\'e}tale map from $\cM(C,Q,L,\partial)$ to   the space $X^\hash$. 
  The isomonodromy connection on the map $\pi_1$ is a flat symplectic connection  whose leaves consist of connections $(E,\nabla_\epsilon)$ with fixed monodromy. Pulling this connection through \eqref{biggy} gives a family of non-linear symplectic connections  $h_\epsilon$ on the projection $\pi\colon X^\hash\to M$. 
  This  gives  rise to a  meromorphic Joyce structure on $M$, with the poles arising because 
  $\alpha$ is only birational rather than an isomorphism, and $\beta_\epsilon$ is only generically  {\'e}tale.

\subsection{Twistor space}

Consider the twistor space $p\colon Z^\hash\to \bP^1$ associated to the Joyce structure of Section \ref{con}. It could be an interesting  problem to try to give a direct moduli-theoretic construction of this space. For now we can at least describe the  fibre $Z_1$ up to an {\'e}tale cover.  Note that since these Joyce structures are meromorphic, the connections $h_\epsilon$ are only well-defined on the projection $\pi\colon X^0\to M$ of an open subset $X^0\subset X=T_M$. We define the twistor fibres as the spaces of leaves of the resulting foliations of  $X^0$. 

%
Choose a reference  surface $S_g$ of genus $g$, set $G=\PGL_2(\bC)$, and define
\begin{equation}\MCG(g)=\pi_0(\operatorname{Diff}^+(S_g)),\qquad \cX(g)= \Hom_{\rm grp}(\pi_1(S_g),G)/G.\end{equation}
Then the mapping class group $\MCG(g)$ acts on the character stack $\cX(g)$ in the usual way, and sending a quadruple $(C,E,\nabla)$ to the monodromy of the connection $\nabla$ defines a map \begin{equation}
\label{stu}\mu\colon \cM(C,E,\nabla)/J^2(C)\to \cX(g)/\MCG(g),\end{equation}
whose fibres are by definition the horizontal leaves of the isomonodromy connection. 
Transferring this map across the diagrams \eqref{biggy}, and passing to the leaf space yields an {\'e}tale map
\begin{equation}
\label{stuck}\mu\colon Z_\epsilon^{\hash}\to \cX(g)/\MCG(g).\end{equation}

It would be interesting to test the  speculations of Section \ref{cancor} in this context. It has been suggested that the appropriate Darboux co-ordinates to consider on the character variety $\cX(g)$ are the Bonahon-Thurston shear co-ordinates \cite{bon,fen}.
Let us instead consider spaces of quadratic differentials with poles of fixed orders.  The corresponding Joyce structures will appear in \cite{Z}.
It is natural to expect that there is a similar map to \eqref{stuck} in which $\cX(g)$ is replaced with a  space of framed local systems \cite{FG}. 
 It is then expected that the preferred Darboux co-ordinates associated to a general  half-plane $H(\varphi)$ are  the Fock-Goncharov co-ordinates for  the WKB triangulation determined by the horizontal trajectories of the quadratic differential $ e^{-2\pi i \varphi}\cdot Q$.
 It follows from the work of Gaiotto, Moore and Neitzke \cite{GMN2}, and the results of \cite{BS}, that as $\varphi$ varies these satisfy the jumps \eqref{st} determined by the  DT theory of the corresponding category. The asymptotic property \eqref{asym} should follow from existing results in exact WKB analysis.  
 The whole story has been treated in full detail \cite{A2} in the special case  discussed in Section \ref{a2} below.

\subsection{Hamiltonian systems}
\label{fab}

Let us consider the Hamiltonian system of Theorem \ref{inta} in the case of  Joyce structures of Section \ref{con}.
Recall that a crucial feature of the construction of these Joyce structures  is the  isomonodromy connection  on the map
\begin{equation}\label{map}\pi_1\colon \cM(C,E,\nabla_1)\to \cM(C),\end{equation}
which is both flat and symplectic.
To construct a Hamiltonian system  we need a whole one-parameter family of such connections. In the notation of Section \ref{ham}, the isomonodromy connection is $k_1$, but we also need  the connection $k_\infty$ before we can write \eqref{hamil}. This issue is often a little hidden in the literature because in many examples 
there is a natural choice for the reference connection $k_\infty$ which is then taken without further comment. 
It is discussed explicitly in \cite{Hu2}, and is also mentioned for example in \cite[Remark 7.1]{Boalch2}.

One way to try to define  a Hamiltonian system   on the map \eqref{map} is to choose for each bundle $E$ a distinguished `reference' connection $\nabla_\infty$. Then we can define a Higgs field $\Phi=\nabla_\infty-\nabla_1$ and for $\epsilon \in \bC^*$ a connection $\nabla_\epsilon=\nabla_\infty-\epsilon^{-1}\Phi$. We  can then define $k_\epsilon$ to be the isomonodromy connection for the family of connections $(E,\nabla_\epsilon)$. Many interesting examples of isomonodromic systems in the literature involve bundles   with meromorphic connections on a genus 0 curve. Since the generic such bundle $E$ is  trivial, it is then natural to take $\nabla_\infty=d$. But for bundles with connection on higher genus curves there is no such canonical choice. 

Consider the meromorphic Joyce structures of Section \ref{con}. Note that the base $M=\cM(C,Q)$ has a natural cotangent bundle structure, with $B=\cM(C)$ being the moduli space of curves of genus $g$, and  $\rho\colon M\to B$  the obvious projection $\rho\colon \cM(C,Q)\to \cM(C)$. Indeed, the tangent spaces to $\cM(C)$ are the vector spaces $\cT_{C} \cM(C)=H^1(C,T_C)$, and Serre duality gives $H^0(C,\omega_C^{\tensor 2})=H^1(C,T_C)^*$. Thus $T_B^*$ parameterises pairs $(C,Q)$ of a curve $C$ together with a quadratic differential $Q\in H^0(C,\omega_C^{\tensor 2})$, and  $M\subset T_B^*$ is the open subset where $Q$ has simple zeroes. 

Let us choose a  Lagrangian $R\subset \cX(g)/\MCG(g)$.  Pulling back by the {\'e}tale map gives a Lagrangian in $Z_\infty$  which we also denote by $R$.  The subspace of the quotient $\cX(g)/\MCG(g)$ consisting of monodromy representations of connections on a fixed bundle $E$ is known to be Lagrangian.   Then for a generic bundle $E$ we can expect these two Lagrangians to meet in a finite set of points, and so locally on the moduli of bundles we can define $\nabla_\infty$ by insisting that its monodromy  lies in $R$. 

Using the natural cotangent bundle structure $\rho\colon M\to B$ and the Lagrangian $R\subset Z_\infty$ we can now apply the construction of Section \ref{hamham}  to these examples.  
We obtain a diagram

 \begin{equation}
\xymatrix@C=1.5em{ 
 Y=q_\infty^{-1}(R)\ \ar@{^{(}->}[rr]^{i}\ar[drr]_{f}
 && \cM(C,E,\nabla,\Phi)\ar@{->}[rr]^{\beta_1}\ar[d]
 &&\cM(C,E,\nabla_1)\ar[dll]^{\pi_1}  \\
 &&  \cM(C)}
 \end{equation}
 where the map $\beta_1$ is defined by setting $\nabla_1=\nabla-\Phi$.
 
 The isomonodromy connection defines a flat connection $h_1$ on the map $\pi_1$. The  Hamiltonian  system of Theorem \ref{inta} defines a whole pencil of flat connections $k_\epsilon$ on the map $f$. The transversality assumption ensures that  the map $\beta_1\circ i$ is {\'e}tale, and  the pullback of the connection $h_1$ then coincides with $k_1$. Thus by choosing the Lagrangian $R\subset Z_\infty$ and using it to define reference connections, we have upgraded the isomonodromy connection   to a Hamiltonian system.


\section{The doubled A$_1$ example}

\label{a1}

In this section we discuss a simple and rather degenerate  example of a pre-Joyce structure which does not quite satisfy the conditions to be a Joyce structure. It is related to the  
 DT theory of the A$_1$ quiver. 

\subsection{Pre-Joyce structure}

We take $M=\bC^*\times \bC$ with co-ordinates $(z,z^\vee)$, and consider the total space of the tangent bundle $X=T_M$ with corresponding  co-ordinates $(z,z^\vee, \theta,\theta^\vee)$ as in Section \ref{prejoyce}.
We define a pre-Joyce structure on $M$ by taking the symplectic form \begin{equation}
\omega=\frac{1}{2\pi i} \cdot dz\wedge dz^{\vee},
\end{equation} and the pencil of connections $h+\epsilon^{-1} v$ defined by
\begin{equation}\label{connie}h_\epsilon\bigg(\frac{\partial}{\partial z}\bigg)=\frac{\partial}{\partial z}+\frac{1}{\epsilon}\cdot \frac{\partial}{\partial \theta}+\frac{\theta}{2\pi i z}\cdot \frac{\partial}{\partial \theta^\vee}, \qquad h_\epsilon\bigg(\frac{\partial}{\partial z^\vee}\bigg)=\frac{\partial}{\partial z^\vee}+\frac{1}{\epsilon}\cdot \frac{\partial}{\partial \theta^\vee}.\end{equation}
 The corresponding  Pleba{\'n}ski function is
\begin{equation}W(z,z^\vee,\theta,\theta^\vee)=-\frac{\theta^3}{6(2\pi i)^2 z}.\end{equation}
 The formulae \eqref{simple} - \eqref{w} become
 
 \begin{gather}\Omega_+=\frac{1}{2\pi i} \cdot dz\wedge dz^{\vee}, \qquad 2i\Omega_I=\frac{1}{2\pi i} \big(d\theta\wedge dz^\vee-d\theta^\vee\wedge dz\big),\\
 \Omega_-=\frac{1}{2\pi i} \cdot d\theta\wedge d\theta^\vee -\frac{\theta}{(2\pi i)^2 z} \cdot d\theta\wedge dz.\end{gather}
 The above pre-Joyce structure was extracted from the DT theory of the A$_1$ quiver in \cite[Section 8]{RHDT2} by first  applying a doubling procedure and then solving the resulting Riemann-Hilbert problem. The doubling procedure is required because the Euler form of the category is degenerate, and in fact identically zero.

 \subsection{Period structure}
 We can define a period structure on $M$ by declaring the co-ordinates $(z,z^\vee)$ to be integral linear. The resulting lattice $T_M^{\bZ}\subset T_M$ is spanned by integral linear combinations of the vector fields  $\frac{\partial}{\partial z}$ and $ \frac{\partial}{\partial z^{\vee}}$. This is a homogeneous period structure, since the vector field 
 \begin{equation}Z=z \cdot \frac{\partial}{\partial z} + z^\vee \cdot \frac{\partial}{\partial z^\vee}\end{equation}
 generates the $\bC^*$-action  $t\cdot (z,z^\vee)=(tz,tz^\vee)$. The  inverse $\eta\colon T_M^*\times T_M^*\to \O_M$ of the symplectic form $\omega$ satisfies $\eta(dz^\vee,dz)=2\pi i$.
 
 Combining the above pre-Joyce structure and period structure does not quite yield a Joyce structure. The problem is with axiom (J2): the  connection \eqref{connie} is not invariant under translations by the lattice $(2\pi i)\, T_M^{\bZ}$ since it depends on $\theta$ rather than its exponential. This problem seems to be related to the doubling procedure referred to above.

\subsection{Alternative co-ordinates}
 There is another integral affine structure on $M$ with integral affine co-ordinates
\begin{equation}v=z, \qquad v^\vee=z^{\vee}-\frac{z}{2\pi i}\bigg(\log\Big(\frac{z}{2\pi i}\Big)-1\bigg).\end{equation}
The associated co-ordinates $(v,v^\vee,\phi,\phi^\vee)$ on $X=T_M$ are
\begin{equation}\phi=\theta,\qquad \phi^\vee=\theta^\vee-\frac{\theta}{2\pi i} \cdot  \log\Big(\frac{z}{2\pi i}\Big).\end{equation}
In this co-ordinate system the  connection $h_\epsilon$ takes the form
\begin{equation}\label{lakes}h_\epsilon\bigg(\frac{\partial}{\partial v}\bigg)=\frac{\partial}{\partial v}+\frac{1}{\epsilon}\cdot \frac{\partial}{\partial \phi}, \qquad h_\epsilon\bigg(\frac{\partial}{\partial v^\vee}\bigg)=\frac{\partial}{\partial v^\vee}+\frac{1}{\epsilon}\cdot \frac{\partial}{\partial \phi^\vee},\end{equation}
and hence the corresponding Pleba{\'n}ski function $W(v,v^\vee,\phi,\phi^\vee)=0$.

Note however that since 
\begin{equation}
    Z= v\cdot \frac{\partial}{\partial v}+ \Big(v^\vee-\frac{v}{2\pi i}\Big)\cdot \frac{\partial}{\partial v^\vee},
\end{equation}
we cannot combine this integral affine structure with the vector field $Z$ to form a period structure.

\subsection{First Pleba{\'n}ski function}
It follows  from \eqref{lakes} that the fibre co-ordinates $(\phi,\phi^\vee)$ are preserved by the connection $h=h_\infty$. By \eqref{w} they  satisfy 
\begin{equation}\Omega_-=\frac{1}{2\pi i}\cdot d\phi\wedge d\phi^\vee.\end{equation}
They therefore descend to Darboux co-ordinates on the twistor fibre $Z_\infty$.
The corresponding first Pleba{\'n}ski function defined as in  Section \ref{u} is
\begin{equation}U(z,z^\vee,\phi,\phi^\vee)=\frac{1}{2\pi i} \bigg(z^\vee \phi - \phi^\vee z +\frac{\phi  z}{2\pi i}-\frac{\phi z}{2\pi i}\cdot  \log\Big(\frac{z}{2\pi i}\Big)\bigg).\end{equation}
It satisfies
\begin{equation}\frac{\partial U}{\partial z^\vee}=\frac{\theta}{2\pi i}, \qquad \frac{\partial U}{\partial z}=\frac{-\theta^\vee}{2\pi i}.\end{equation}
The Jacobian matrix has entries
\begin{equation}\label{jack}\frac{\partial^2 U}{\partial z \partial \phi}=-\frac{1}{2\pi i} \log\Big(\frac{z}{2\pi i}\Big), \qquad \frac{\partial^2 U}{\partial z^\vee \partial \phi}=\frac{1}{2\pi i}=-\frac{\partial^2 U}{\partial z \partial \phi^\vee}, \qquad \frac{\partial^2 U}{\partial z^\vee \partial \phi^\vee}=0,\end{equation}
and  has determinant $-1/(2\pi i)^2$  as required by the first Pleba{\'n}ski equation \eqref{firstpleb}.

\begin{remark}
In this case the  Jacobian matrix \eqref{jack} is skew-symmetric. This means we can find $\cF=\cF(z,z^\vee)$ such that
\begin{equation}\frac{\partial \cF}{\partial z}=\frac{\partial U}{\partial \phi}, \qquad \frac{\partial \cF}{\partial z^\vee}=-\frac{\partial U}{\partial \phi^\vee}.\end{equation}
Then $\cF$ is closely related to the prepotential of \cite[Section 7]{RHDT2}. Explicitly we have
\begin{equation}\cF(z,z^\vee)=\frac{z z^\vee}{2\pi i}- \frac{1}{(2\pi i)^2} \bigg(\frac{1}{2}z^2\log \Big(\frac{z}{2\pi i}\Big)-\frac{3}{4}z^2\bigg).\end{equation}
\end{remark}


\section{ The A$_2$ example}
\label{a2}

Here we consider a very interesting Joyce structure which is perhaps the simplest non-trivial example. It is related to the  DT theory of the A$_2$ quiver, and is an example of a Joyce structure of class $S[A_1]$. For more details on this example we refer the reader to   \cite{A2}. 

\subsection{Period structure}The base of the Joyce structure is
\begin{equation}M=\{(a,b)\in \bC^2:  4a^3+27b^2\neq 0\}.\end{equation}
Associated to a pair $(a,b)\in M$ is a  quadratic differential
 \begin{equation}
\label{coffee}Q_0(x)\, dx^{\tensor 2}=(x^3+ax+b) \, dx^{\tensor 2}\end{equation} 
 on $\bP^1$ which has a single pole of order 7 at $x=\infty$  and simple zeroes. There is an associated   double cover $\Sigma\to \bP^1$ which  is the projectivization of the affine elliptic curve $y^2=x^3+ax+b$.
 
 Take a basis of cycles  $(\gamma_1,\gamma_2)\subset H_1(\Sigma,\bZ)$ with intersection $\gamma_1\cdot \gamma_2=1$. We shall need the periods and quasi-periods of the elliptic curve $\Sigma$. They are given by
\begin{equation}\omega_i=\int_{\gamma_i} \frac{dx}{2y}, \qquad \eta_i =-\int_{\gamma_i} \frac{x\, dx}{2y},\end{equation}
and satisfy the Legendre relation $\omega_2\eta_1-\omega_1\eta_2=2\pi i$. 

There are local co-ordinates on $M$ given by
\begin{equation}
\label{zz}z_i=\int_{\gamma_i} y\, dx=\int_{\gamma_i} \sqrt{x^3+ax+b\,} \, dx,\end{equation}
and relations 
\begin{equation}
\label{pens1}\frac{\partial}{\partial a}=-\eta_1\frac{\partial}{\partial z_1} -\eta_2\frac{\partial}{\partial z_2},\qquad \frac{\partial}{\partial b}=\omega_1\frac{\partial}{\partial z_1} +\omega_2\frac{\partial}{\partial z_2}.\end{equation}

We define a period structure on $M$ by declaring $(z_1,z_2)$  to be the integral linear co-ordinates. We take the  symplectic form 
 \begin{equation}\omega=-\frac{1}{2\pi i} dz_1\wedge dz_2 = da\wedge db.\end{equation}
 Note that the inverse satisfies
 \begin{equation}\eta\big(dz_1,dz_2\big)=2\pi i\end{equation}
 and so axiom (J1) of Definition \ref{joyce} is satisfied.

 The Euler vector field is
 \begin{equation}Z=z_1\frac{\partial}{\partial z_1}+z_2\frac{\partial}{\partial z_2}= \frac{4a}{5} \frac{\partial}{\partial a} + \frac{6b}{5}  \frac{\partial}{\partial b}.\end{equation}
 
 Note that $Z$  does not generate a $\bC^*$ action on $M$. In fact the moduli space of quadratic differentials of the form \eqref{coffee} is the quotient $M/\mu_5$, where the group $\mu_5\subset \bC^*$ acts via $\zeta\cdot (a,b)= (\zeta^4a,\zeta^6b)$. The vector field $Z$ generates an action of $\bC^*$ on this quotient. 
 
 \subsection{Joyce structure}
 
 We have local co-ordinates $(z_i,\theta_j)$ on the tangent bundle $X=T_M$. We consider the quotient $X^\hash=T_M/(2\pi i)\, T_M^{\bZ}$. The fibre over a point $(a,b)\in M$ is the cohomology group $H^1(\Sigma,\bC^*)\isom (\bC^*)^2$. We introduce some alternative co-ordinates on $X^\hash$ by the relation
 \begin{equation}
\label{xi}\theta_i=-\int_{\gamma_i} \bigg(\frac{p}{x-q} + r\bigg) \frac{dx}{2y},\end{equation}
where $p^2=q^3+aq+b$ and $r\in \bC$. 
The integral is well-defined up to multiples of $2\pi i$. The numbers $\xi_i=\exp(\theta_i)$  are the holonomies of a  connection on the line bundle on $\Sigma$ with divisor $(q,p)-\infty$.

In terms of the  fibre co-ordinates $(\theta_a,\theta_b)$ associated to the co-ordinates $(a,b)$ we then have
\begin{equation}
\label{hide}\theta_i=-\eta_i\theta_a+\omega_i \theta_b.\end{equation}
 The functions $(\theta_a,\theta_b)$  are  given explicitly by
\begin{equation}\label{choc}\theta_a=-\frac{1}{4}\int^{(q,p)}_{(q,-p)} \frac{dx}{y}, \qquad \theta_b=\frac{1}{4}\int^{(q,p)}_{(q,-p)} \frac{x dx}{y}-r.\end{equation}

The Joyce structure on $X=T_M$ is obtained by taking the  connection $h_\epsilon$  given by

\begin{align}\label{da} h_\epsilon\bigg(\frac{\partial}{\partial a}\bigg)&=-\frac{2p}{\epsilon} \frac{\partial}{\partial q}
-\frac{q}{\epsilon} \frac{\partial}{\partial r} +\bigg(\frac{\partial}{\partial a}-\frac{r}{p}\frac{\partial}{\partial q}
-\frac{r^2(3q^2+a)-qpr}{2p^3}  \frac{\partial}{\partial r}\bigg),\\
\label{db} h_\epsilon\bigg(\frac{\partial}{\partial b}\bigg)&=-\frac{1}{\epsilon}\frac{\partial}{\partial r}+\bigg(\frac{\partial}{\partial b}
 +\frac{r}{2p^2} \frac{\partial}{\partial r}\bigg).\end{align}

This is the isomonodromy connection for a 
pencil of connections $\nabla-\epsilon^{-1}\Phi$ on  the trivial rank 2 bundle on $\bP^1$ with
\begin{equation}\label{conn}\nabla=d-\mat{r}{0}{0}{-r} \frac{dx}{2p}, \qquad \Phi=\mat{p}{x^2+xq+q^2+a}{x-q}{-p} dx.\end{equation}
 Equivalently we can consider the deformed cubic oscillator
\begin{equation}
\label{coffeee!}y''(x)=Q(x)y(x), \qquad Q(x)=\epsilon^{-2}Q_0(x) + \epsilon^{-1} Q_1(x) + Q_2(x),\end{equation}
where the terms in the potential are
\begin{equation}Q_1(x)=\frac{p}{x-q}+r, \qquad
Q_2(x)=\frac{3}{4(x-q)^2}+\frac{r}{2p(x-q)}+\frac{r^2}{4p^2}.\end{equation}

There is a rational expression for the  Pleba{\'n}ski  function
\begin{equation}W= \frac{1}{4(4a^3+27b^2)p} \Big(2apr^3-(6aq^2-9bq+4a^2)r^2-3p(3b-2aq)r-2ap^2\Big),\end{equation}
although note that this does not satisfy the second of the normalisation conditions \eqref{fix}. In other words, the uniquely-defined Pleba{\'n}ski function of Lemma \ref{plebe} differs from the above expression by a function of the form $\sum_i a_i(z)\cdot  \theta_i$.

\subsection{Further properties}


In the co-ordinates $(a,b,q,r)$ the Euler vector field $E$ is
\begin{equation}E=\frac{4a}{5} \frac{\partial}{\partial a} + \frac{6b}{5}  \frac{\partial}{\partial b}+ \frac{2q}{5} \frac{\partial}{\partial q}+ \frac{r}{5} \frac{\partial}{\partial r}.\end{equation}
The expressions \eqref{pens1} and \eqref{hide} and a short calculation using \eqref{choc} gives
\begin{equation}
\label{flour}2i\Omega_I=-da \wedge d\theta_b+db\wedge d\theta_a =dq\wedge dp+da\wedge dr.\end{equation}

The twistor fibre $Z_\infty$ is the space of leaves of the foliation defined by the connection $h=h_\infty$. The functions 
\begin{equation}\phi_1=q+\frac{ar}{p}, \qquad \phi_2=\frac{r}{2p}, \end{equation}
descend to  $Z_\infty$
because they are constant along the  flows \eqref{da} and \eqref{db} with $\epsilon=\infty$. Moreover, if we restrict to a fibre $F$ of the projection $\pi\colon X=T_M\to M$  by fixing $(a,b)$, then 
 \begin{equation}-\frac{1}{2\pi i}\cdot d\theta_1\wedge d\theta_2\big|_F=d\theta_a \wedge d\theta_b\big|_F=-\frac{dr}{2p}\wedge dq\big|_{F}=d\phi_1\wedge d\phi_2\big|_F.\end{equation}
 Thus we have
 \begin{equation}\Omega_\infty=d\phi_1\wedge d\phi_2.\end{equation}

 The first Pleba{\'n}ski function $U\colon X\to \bC$ of Appendix \ref{u} is   given by
\begin{equation}\label{analo}U=\frac{1}{2} \int_{(q,-p)}^{(q,p)} (x^3+ax+b)^{1/2}\, dx,\end{equation}
because a simple calculation using the flows \eqref{da} and \eqref{db} shows that
\begin{equation}\frac{\partial U}{\partial b}\bigg|_{\phi}=h\Big(\frac{\partial}{\partial b}\Big) U=-\theta_a, \qquad \frac{\partial U}{\partial a}\bigg|_{\phi}=h\Big(\frac{\partial}{\partial a}\Big) U=\theta_b.\end{equation}

 Since the $\bC^*$-action rescales $\phi_1$ and $\phi_2$ with weights $- \tfrac{2}{5}
$ and $\tfrac{2}{5}$ respectively, we have
\begin{equation}i_E(\Omega_\infty)=i_E(d\phi_1\wedge d\phi_2)=\frac{2}{5} d (\phi_1\phi_2).\end{equation}
The Joyce function $F$ of Section \ref{joy} is given by
\begin{equation} F=\frac{1}{5} (1-2\phi_1\phi_2),\end{equation}
where for  the constant normalisation we used the limiting behaviour of $q,p,r$ along the zero section $M\subset T_M$ as discussed in \cite[Section 4.3]{A2}.

The distinguished fixed point of $Z_\infty$ is defined by $(\phi_1,\phi_2)=(0,\infty)$. It is an isolated fixed point. The linear Joyce connection on $M$ is the one whose flat co-ordinates are $(a,b)$, and the negative and positive weight spaces $V_-$ and $V_+$ of Lemma \ref{lem} are spanned by $\frac{\partial}{\partial a}$ and $\frac{\partial}{\partial b}$ respectively. The Joyce metric is
\begin{equation}g=\frac{1}{5} \cdot (da\tensor db+ db\tensor da).\end{equation}

\end{document}